\documentclass{amsart}
\usepackage{amssymb, amsmath}
\usepackage{color}
\usepackage{enumerate} 
	
\let\d=\partial

\newcommand{\loc}{{\rm loc}\,}

\def\C{\mathop{\mathbb C\kern 0pt}\nolimits}
\def\R{\mathop{\mathbb R\kern 0pt}\nolimits}
\def\N{\mathop{\mathbb N\kern 0pt}\nolimits}

\def\dy{\,{\rm d}y\,}

\def\dx{\,{\rm d}x\,}

\def\dt'{\,{\rm d}t'\,}

\def\ds'{\,{\rm d}s'\,}

\def\tanh{\,{\rm tanh}\,}

\def\cL{\mathcal{L}}

\def\cS{\mathcal{S}}

\newtheorem{ex}{Example}[section]
\newtheorem{thm}{Theorem}[section]
\newtheorem{lem}{Lemma}[section]
\newtheorem{rmk}{Remark}[section]
\newtheorem{col}{Corollary}[section]

\begin{document}

\title[Eigenvalues for the Lax operator]{Eigenvalue analysis of the Lax operator for the one-dimensional cubic nonlinear defocusing Schr\"odinger equation}

   \author[X. Liao]{Xian Liao}
\address[X. Liao]%
 {Institute for Analysis, Karlsruhe Institute of Technology, Englerstrasse 2, 76131 Karlsruhe, Germany}
\email{xian.liao@kit.edu} 
 
\author[M. Plum]{Michael Plum}
\address[M. Plum]%
 {Institute for Analysis, Karlsruhe Institute of Technology, Englerstrasse 2, 76131 Karlsruhe, Germany}
\email{michael.plum@kit.edu} 


  \begin{abstract}
     We characterize the location and   number of eigenvalues for the Lax operator associated to the 
     one-dimensional cubic nonlinear defocusing Schr\"odinger equation. 
     With the help of a newly discovered unitary matrix, the analysis reduces to the study of the spectral problem for a unitarily equivalent operator, which involves only the amplitude and the phase velocity of the potential.
     Examples of potentials with special amplitude and phase velocity are   investigated.
 \end{abstract}
 
\maketitle

\noindent{\sl Keywords:} Cubic nonlinear defocusing Schr\"odinger equation, nonzero boundary condition, Lax operator, one-dimensional Dirac operator, Sturm-Liouville eigenvalue problem, spectral analysis

\vspace{.1cm}

\noindent{\sl AMS Subject Classification (2020):} 35Q55, 37K10\
 
 \bigskip
 
 \section{Introduction}
 We consider the following one-dimensional defocusing cubic nonlinear  Schr\"odinger (NLS) equation
\begin{equation}\label{GP}
i\d_t q+\d_{xx} q=2|q|^2q,
\end{equation} 
where $q=q(t,x):\R\times \R\to \C$ denotes the unknown wave function.
By the seminal paper by Zakharov-Shabat  \cite{ZS71}, the (NLS) equation can be (formally) formulated in the Lax pair form 
\begin{equation}\label{LaxPair}
\d_t L=PL-LP,
\end{equation}
where $L$ denotes the self-adjoint Lax operator   with the potential $q$
\begin{equation}\label{L}
L=L_q=\left(\begin{matrix}
i\d_x&-iq \\ i\bar q&-i\d_x
\end{matrix}\right),
\end{equation} 
and $P$ is the following skewadjoint differential operator 
\begin{equation*}\label{P,GP}
P=i\left(\begin{matrix}
2\partial_x^2-|q|^2&-q\partial_x-\partial_x q
 \\
  \bar q\partial_x+\partial_x\bar q & -2\partial_x^2+|q|^2
\end{matrix}\right).
\end{equation*}
Here the application of the operator $\d_x\bar q$ on a function $f$ is understood as $\d_x(\bar q f)$.
Let $U(t',t)$ be the unitary family generated by the skewadjoint operator $P$, then by virtue of \eqref{LaxPair}, one can relate the operators   $L(t):=L_{q(t,x)}$ and $L(t')=L_{q(t',x)}$ at different times by
$$
L(t)=U^\ast(t', t)L(t')U(t',t),
$$
such that  the spectrum of the Lax operator $L(t)$ is (formally) invariant under the evolutionary NLS-flow \eqref{GP}. 
In particular, the eigenvalues of $L(0)$ at the initial time $t=0$ are   the eigenvalues of $L(t)$ for all the time $t$ (as long as the solution   exists).
In the present paper we will analyze the eigenvalues of the Lax operator $L_q$ for a class of   nowhere vanishing bounded  potentials $q$, and in the following we will largely ignore the time dependence.

In the (classical) setting of decaying potentials: 
$$
q(x)\rightarrow 0\hbox{ as }|x|\rightarrow\infty,
$$
the spectral problem of the Lax operator $L_q$ and the associated direct/inverse scattering transform 
have been  extensively studied  in the literature, cf. the book \cite{AS}.
If one assumes   the nonzero boundary condition   for $q$ at infinity:
\begin{equation}\label{BC}
    |q(x)|\rightarrow 1 \hbox{ as }|x|\rightarrow\infty,
\end{equation} 
the equation \eqref{GP} possesses a family of soliton solutions  $e^{-2it} q_c(x-2ct)$ where
$$
q_c(x)= \sqrt{1-c^2}\,\tanh\bigl(\sqrt{1-c^2}\,x\bigr)+ic ,\quad -1<c<1,
$$ 
and the corresponding Lax operator $L_{q_c}$ has a unique simple eigenvalue   $-c$.
These soliton solutions are called dark/black solitons in nonlinear optics, which travels at the speed $2|c|<2$, and there are no soliton solutions with traveling speed bigger than $2$.
Due to the experimental relevance of the problem \eqref{GP}-\eqref{BC}, the study of the spectral problem of the Lax operator $L_q$ under the assumption  \eqref{BC}  (i.e. the one-dimensional Dirac operator with nonzero rest mass) has   attracted much attention,  cf.  \cite{AKNS, BP, CJ, DZ, DPVV, FT, GZ, ZS73}. 
In particular in the classical framework where
\begin{equation*}\label{FT,BC}
q(x)-1\in  \cS(\R) 
\end{equation*} 
is a Schwartz function, Faddeev-Takhtajan  \cite{FT} studied  the self-adjoint operator $L_q$, and showed that its essential spectrum is $(-\infty, -1]\cup [1,\infty)$ and there are at most countably many \emph{simple real} eigenvalues $\{\lambda_m\}$  inside $(-1,1)$.  
More recently, Demontis et al. \cite{DPVV} studied rigorously the inverse scattering transform if
\begin{align*}
    q(x)\rightarrow e^{i\theta_\pm}\in \mathbb{S}^1\hbox{ as }x\rightarrow\pm\infty 
\end{align*}
sufficiently fast in the sense that $(1+x^2)(q(x)-e^{i\theta_\pm})\in L^1(\R^\pm)$.
Under the stronger decay assumption
\begin{equation}\label{potential,decay}
    (1+x^4)(q(x)-e^{i\theta_\pm})\in L^1(\R^\pm),
\end{equation} 
they showed  that there are only finitely many discrete eigenvalues which belong to the spectral gap $(-1,1)$.  
It was shown recently in \cite{KL:low} that in the low-regularity finite-energy setting
$$q(x)\in L^2_\loc(\R)\hbox{ with }|q|^2-1, \d_xq\in H^{-1}(\R),$$
  the essential spectrum of the Lax operator $L_q$ is $(-\infty, -1]\cup [1,\infty)$, and   the spectrum outside the essential spectrum
consists of isolated simple eigenvalues  in $(-1, 1)$.
However, under this weak assumption, there  might be eigenvalues embedded in the essential spectrum.
A specific kind of piecewise constant potentials 
\begin{equation*}\label{BPpotential} 
    q(x)=\left\{\begin{array}{cc}
       e^{i\theta_-}  & x<-R,  \\
       Ae^{i\varphi}  & -R<x<R,\\
       e^{i\theta_+} & x>R,
    \end{array}\right.
 \end{equation*}
has been considered in \cite{BP}, and the authors there estimated the location of the discrete eigenvalues inside the spectral gap $(-1,1)$ by considering the relation between $A$ and $\cos(\varphi)$. In particular, if $A<1$, then there is at least one discrete eigenvalue.
Finally we mention   here a recent work \cite{BFP} assuming the nonzero asymmetric boundary condition on the potential
$$
q(x)\rightarrow q_\pm,\hbox{ as }x\rightarrow\pm\infty,
\hbox{ with }|q_+|\neq |q_-|.
$$ 

In spite of its physical and mathematical relevance, the spectral theory for the Lax operator $L=\left(\begin{matrix}
i\d_x&-iq \\ i\bar q&-i\d_x
\end{matrix}\right)$ with general potentials $q$ is still far from satisfactory.
Here we propose a  new idea to study the operator $L$, when the potentials are assumed to be   nowhere vanishing, bounded and with finite phase velocity:
\begin{align}\label{Assumption}
    q=|q|e^{i\varphi}\in L^\infty(\R;\C),\quad |q|>0, \quad \d_x\varphi\in L^\infty(\R;\R).
\end{align} 
By straightforward calculations, the following   unitary matrix which we believe to be new
$$M=\frac1{\sqrt2}\begin{pmatrix} e^{-\frac12i(\varphi-\frac\pi2)} & e^{\frac12i(\varphi-\frac\pi2)}\\
 e^{-\frac12i(\varphi-\frac\pi2)}&-e^{\frac12i(\varphi-\frac\pi2)}\end{pmatrix}: H^s(\R;\C^2)\to H^s(\R;\C^2),\quad s=0,1,$$
 transforms the Lax operator $L: H^1(\R;\C^2)\to L^2(\R;\C^2)$   to the following unitarily equivalent operator 
\begin{align}\label{cL,intro}
    \mathcal{L}=MLM^\ast =\left( \begin{matrix}
 -u_- &  i\d_x \\  i\d_x & -u_+
\end{matrix} \right): H^1(\R;\C^2)\to L^2(\R;\C^2),
\end{align}
where the two real-valued functions $u_\pm$ read 
\begin{equation}\label{upm,intro}
u_\pm:=\frac12\d_x\varphi\pm |q|\in L^\infty(\R;\R).
\end{equation}
This proves
\begin{lem}[Unitary equivalence between   $L$ and   $\cL$]\label{lem:L-cL}
 For nowhere vanishing bounded potentials $q$ with finite phase velocity such that $u_\pm\in L^\infty(\R;\R)$,   the operator $L: H^1(\R;\C^2)\to L^2(\R;\C^2)$ and   the operator $\cL: H^1(\R;\C^2)\to L^2(\R;\C^2)$ 
 are unitarily equivalent:
 $$
 L=M^\ast \cL M. 
 $$
 \end{lem}
Hence it suffices to study the spectral problem for $\cL$. 
Simply by integration by parts, one can  show that for any $c\in\R$,  $u_-+c$ (resp. $u_+-c$) controls the size $c-\lambda$ (resp. $c+\lambda$), where $\lambda$ is any eigenvalue of $\cL$, in the following sense: 
\begin{thm}[Location of  eigenvalues of $\cL$]\label{thm:cL}
Let $u_\pm\in L^\infty(\R; \R)$.
If $\lambda\in\R$ is an eigenvalue of the operator 
$\cL: H^1(\R;\C^2)\to L^2(\R;\C^2)$ given in \eqref{cL,intro},
 then $\lambda$ satisfies, for all $c\in\R$, 
 \begin{align*}
     c-\lambda\leq \|(u_-+c)^{+}\|_{L^\infty}\quad
     \hbox{ or }\quad c+\lambda \leq \|(u_+-c)^{-}\|_{L^\infty}.
 \end{align*}
In particular if $u_+, -u_-\geq c>0$,  there are no eigenvalues of  $\cL$ in $(-c,c)$.
 \end{thm}
In the above,  $f^{+}, f^{-}$ denote the positive and negative parts of a real-valued function $f$ respectively.
 Due to the unitary equivalence between  the Lax operator $L$ and the operator $\cL$, we have immediately 
 \begin{col}[Location of  eigenvalues of $L$]\label{col:L} 
 Let $q=q(x)$ satisfy \eqref{Assumption}.
The eigenvalues $\lambda\in\R$ of the operator $L:H^1(\R;\C^2)\to L^2(\R;\C^2)$ given in \eqref{L}   satisfy, for all $c\in\R$, 
 \begin{align*}
      c-\lambda\leq \Bigl\|\Bigl(\frac12\d_x\varphi-|q|+c\Bigr)^{+}\Bigr\|_{L^\infty}
     \quad
     \hbox{ or }\quad c+\lambda \leq \Bigl\|\Bigl(\frac12\d_x\varphi+|q|-c\Bigr)^{-}\Bigr\|_{L^\infty}.
 \end{align*}
 In particular, if $|q|\geq c+\frac12|\d_x\varphi|$ pointwise for some $c>0$, then there are no eigenvalues of  $L$ in  $(-c,c)$.
 \end{col} 
 Theorem \ref{thm:cL} will be  proved in Section \ref{sec:general}.

\smallbreak 

If one assumes  the following boundary condition for the nowhere vanishing bounded potentials $q$ given \eqref{Assumption}:  
\begin{align}\label{BC:q,varraphi}
    |q(x)|\rightarrow 1\hbox{ and }\partial_x\varphi(x)\rightarrow 0\hbox{ as }|x|\rightarrow\infty,
\end{align} 
which is stronger than   \eqref{BC},
then $u_\pm$ given in \eqref{upm,intro} satisfy 
$$u_\pm(x)\rightarrow \pm 1
\hbox{ as }|x|\rightarrow \infty.$$
Under  the smallness condition
\begin{equation}\label{smallness:upm}
     \|(u_+-1)^-\|_{L^\infty}+\|(u_-+1)^+\|_{L^\infty}<2,
\end{equation}
by Theorem \ref{thm:cL}, the eigenvalues $\lambda\in (-1,1)$ of $L,\cL$ are located 
\begin{align*}
    \hbox{ either in }\,I_-:=(-1,-1+\|(u_+-1)^-\|_{L^\infty}]\quad \hbox{ or in  }\, I_+:=[1-\|(u_-+1)^-\|_{L^\infty},1),
\end{align*}
where $I_+$ and $I_-$ are disjoint.
We have the following characterization of the numbers of eigenvalues inside $I_-$ or $I_+$ respectively:
\begin{thm}[Number of  eigenvalues]\label{thm:lambda}
Let $q=q(x)$ satisfy \eqref{Assumption}-\eqref{BC:q,varraphi}-\eqref{smallness:upm}.
The following holds true:
\begin{enumerate}
    \item Let $m\in \N_0\cup\{\infty\}$ denote the number of negative eigenvalues $\mu$ of the Sturm-Liouville eigenvalue problem  
\begin{equation}\label{ev,mu,intro}
    -\d_x\Bigl( \frac{1}{1-u_-}\d_x\psi\Bigr)-(1-u_+)\psi=\mu \psi,\quad x\in\R.
\end{equation} 
    Then there are precisely $m$ eigenvalues $\lambda$ inside $ I_-$ of the operators $L, \cL$, counting by multiplicity.
    \item Let $l\in \N_0\cup\{\infty\}$ denote the number of negative eigenvalues $\nu$ of the Sturm-Liouville eigenvalue problem 
\begin{equation}\label{ev,nu,intro}
     -\d_x\Bigl( \frac{1}{1+u_+}\d_x\psi\Bigr)-(1+u_-)\psi=\nu \psi,\quad x\in\R.
\end{equation}
    Then there are precisely $l$ eigenvalues $\lambda$ inside $I_+$ of the operators $L, \cL$, counting by multiplicity.
\end{enumerate} 
\end{thm}  
  We will prove Theorem \ref{thm:lambda} in Section \ref{sec:GP} by use of the min-max principle. 
In Section \ref{sec:GP} we will also show the (non-)existence of eigenvalues for the Lax operator with some special potentials, as applications of Theorem \ref{thm:cL} and Theorem \ref{thm:lambda}.   
  
  At the end of the introduction part we give some remarks below.
  \begin{rmk}[Eigenvalue problems \eqref{ev,mu,intro} and \eqref{ev,nu,intro}]\label{rmk:RR}
  The eigenvalue problems \eqref{ev,mu,intro} and \eqref{ev,nu,intro} are formulated strongly here. It is however sufficient to consider their weak formulations  as follows
\begin{align*}
&\int_{\R}\Bigl( \frac{1}{1-u_-}\d_x\psi\,\d_x\bar\varphi-(1-u_+)\psi\bar\varphi\Bigr) \dx=\mu \int_{\R} \psi\bar\varphi\dx,\quad \forall\varphi\in H^1(\R),
\\
&\int_{\R}\Bigl( \frac{1}{1+u_+}\d_x\psi\,\d_x\bar\varphi-(1+u_-)\psi\bar\varphi\Bigr) \dx=\nu \int_{\R} \psi\bar\varphi\dx,\quad \forall\varphi\in H^1(\R).
\end{align*}

  If $u_+\geq 1$ (resp. $u_-\leq -1$), then \eqref{ev,mu,intro} (resp. \eqref{ev,nu,intro}) has no negative eigenvalues, and hence Theorem \ref{thm:lambda} implies that all the eigenvalues $\lambda\in (-1,1)$ of the operators $L,\cL$ lie  in $I_+$ (resp. $I_-$).
    This partially recovers  results in Theorem \ref{thm:cL} with $c=1$.
    
  One can use the Rayleigh-Ritz method to show  the existence of negative eigenvalues for   \eqref{ev,mu,intro} or  \eqref{ev,nu,intro}.
 See \cite[Theorem 10.23]{NPW} for more explanations for the  method,
and see Example \ref{ex:ev} in Section \ref{sec:GP} below for the existence result of negative eigenvalues of \eqref{ev,mu,intro} or \eqref{ev,nu,intro} as an application of    the Rayleigh-Ritz method.
\end{rmk}

\begin{rmk} [Functions $u_\pm$]

The   compressible Euler equations 
\begin{equation}\label{Euler}\begin{split}
    \left\{ \begin{array}{l}\d_t \rho+2\d_x(\rho v)=0,
    \\
    \d_t (\rho v)+ 2\d_x(\rho v^2)+\d_x p =0,
    \end{array}\right.
\end{split}\end{equation} 
where $(\rho,v): \R\times\R\to [0,\infty)\times\R$ denote the unknown density and velocity functions respectively, 
are used to describe the motion of  compressible fluids.
For example,   \eqref{Euler} together with the pressure law (up to a constant) 
$p=P_h(\rho):=\rho^2 $
governs the flow of (a class of) polytropic gases,
and in the study of its dynamics, the two Riemannian invariants $u_\pm:=\frac12 v\pm \sqrt{\rho}$   played an important role, cf. the book \cite{Smoller}.

If $q=|q|e^{i\varphi}\neq 0$ everywhere, then  
one can write   the NLS equation \eqref{GP} in its hydrodynamic form \eqref{Euler} for the corresponding density and velocity functions given by  $(\rho,v):=(|q|^2, \d_x\varphi)$, and  the    pressure law
reads  $ p=P_h(\rho)+P_q(\rho)$,
where $P_q(\rho):=-\frac12\rho\,\d_x(\frac{\d_x\rho}{\rho})$ denotes the so-called quantum pressure.
 The corresponding functions $u_\pm=\frac12\d_x\varphi\pm|q|$ given in \eqref{upm,intro}  have also been used to study NLS-related problems in the literature,   e.g. in the study of hydrodynamic optical soliton tunneling in \cite{SHE}, and in the study of the semiclassical limit   in \cite{JLM, Jin}.
 
 Observe that if $u_\pm$ satisfy
\begin{align*}
    -1\leq u_-<u_+\leq 1,\quad u_\pm(x)\rightarrow \pm1\hbox{ as }|x|\rightarrow\infty,\quad  \|1-u_+\|_{L^\infty}+\|1+u_-\|_{L^\infty}<2,
\end{align*}  
then the correspnding potential $q=|q|e^{i\varphi}$ satisfies the assumptions  \eqref{Assumption}-\eqref{BC:q,varraphi}-\eqref{smallness:upm}.
 In \cite{Jin}, under some further strong assumptions such as boundedness, single critical points and decay assumptions on $(u_++u_-)$ and $(u_+-u_--2)$,
 Jin   used the WKB method to calculate the asymptotic number of eigenvalues of $L^{\hbar}:=\left(\begin{matrix}
i\hbar\d_x&-iq_{\hbar} \\ i\bar q_{\hbar}&-i\hbar\d_x
\end{matrix}\right)$, $q_{\hbar}(x)=|A(x)|e^{i\frac{\varphi(x)}{\hbar}}$ in the semiclassical limit $\hbar\rightarrow0$: 
$$
\frac{1}{\pi\hbar} \int_{\R} \Bigl( \sqrt{(1-u_+(x))(1-u_-(x))} +\sqrt{(1+u_+(x))(1+u_-(x))}\Bigr)\dx.
$$
We believe that the unitarily equivalent formulation $\cL$ of $L$ will give new observations to    interesting  NLS-related  problems.
 
  \end{rmk} 
  
  \begin{rmk}[Example of potentials with piecewise-constant amplitude and phase velocity]
We will show that (see Example \ref{ex:2} in Section \ref{sec:GP}  below), for potentials  with piecewise-constant amplitude and phase velocity
\begin{equation}\label{potential,linear,intro} 
    q(x)=\left\{\begin{array}{cc}
       e^{i\theta}  & x< a,  \\
       Ae^{i(\theta+v(x-a))}  & a\leq x\leq b,\\
       e^{i(\theta+v(b-a))} & x> b,
    \end{array}\right.
 \end{equation}
 where  $(a,b)\subset\R$,  $\theta\in [0,2\pi)$,  $A>0$ and $v\in \R$,  the Lax operator $L$ has
 \begin{itemize}
 \item No eigenvalues if $A\geq 1$ and $\frac12|v|\leq A-1$. 
In particular, the case with large amplitude $A\geq 1$ and  constant phase $v=0$ is included: $q=e^{i\theta}\left\{\begin{array}{l}A \hbox{ on }[a,b],\\ 1\hbox{ otherwise.} \end{array}\right.$
     \item Exactly two simple eigenvalues, located in $I_-=(-1,-\frac12v-A]$ and $I_+=[-\frac12v+A,1)$ respectively, if $A\in (0,1)$ and $\frac12|v|<1-A$.
      In particular, the case with small amplitude $A\in (0,1)$ and  constant phase  $v=0$ is included. 
     \item Exactly one simple eigenvalue, located in $I_-=(-1,-\frac12v-A]$, if $A\in (0,1)$ and  $\frac12v\in (-1-A,-1+A]$;
     \item Exactly one simple eigenvalue, located in $I_+=[-\frac12v+A,1)$, if  $A\in (0,1)$ and $\frac12v\in [1-A,1+A)$.
 \end{itemize} 
 
 It is straightforward to check that potentials given in \eqref{potential,linear,intro}: 
 $$q(x)=\bigl( 1+(A-1)1_{[a,b]}(x)\bigr) e^{i\bigl(\theta+v\,\chi(x)\bigr)},$$
 where $1_{[a,b]}$ denotes the characteristic function on $[a,b]$ and $\chi$ denotes the Lipschitz continuous function $  (x-a)1_{[a,b]}+(b-a)1_{(b, \infty)}$, satisfy
 \begin{align*}
      &q\in L^2_{\hbox{\tiny loc}}(\R), \quad |q|^2-1=(A^2-1)1_{[a,b]}\in L^2(\R), \\
      &\d_x q= e^{i (\theta+v\,\chi  )}\bigl( (A-1)(\delta_a-\delta_b)+iAv1_{[a,b]}\bigr)\in H^{-1}(\R),
 \end{align*}
 where $\delta_c$, $c\in \R$ denotes the Dirac function such that $\langle \delta_c, \varphi\rangle_{\mathcal{D}'(\R), \mathcal{D}(\R)}=\varphi(c)$ holds for all test functions $\varphi\in \mathcal{D}(\R)$. 
 Hence by \cite[Theorem 1.5]{KL:low} there exists a unique global-in-time solution of the NLS equation \eqref{GP} with the initial data given in \eqref{potential,linear,intro} (in the sense given there), and by \cite[Theorem 1.6]{KL:low} the corresponding (simple) eigenvalues inside $(-1,1)$ of the Lax operator given above preserve for all the time.
  \end{rmk}

\noindent{\it Organization of the paper.}
We will prove  Theorem \ref{thm:cL} in Section \ref{sec:general}, for  general  nowhere vanishing bounded potentials with finite phase velocity.

In Subsection \ref{subs:number}  we will prove Theorem \ref{thm:lambda}   for    nowhere vanishing bounded potentials with unit-size amplitude and vanishing phase velocity at infinity.
In Subsection \ref{subs:example} we will  give some interesting examples as illustration of Theorem \ref{thm:cL} and Theorem \ref{thm:lambda}.  

In the appendix we will consider  a specific kind of  potentials where $u_+=1$ or $u_-=-1$, and      the eigenvalues will be characterized   via a family of compact operators.
It is also of independent interest.

\section{General case of nowhere vanishing bounded potentials with finite phase velocity}\label{sec:general}
\setcounter{equation}{0}

In this section we will study the eigenvalue problem for the Lax operator  
$$L=\left(\begin{matrix}
i\d_x&-iq \\ i\bar q&-i\d_x
\end{matrix}\right),$$
in the (general) case of nowhere vanishing bounded potentials   with finite phase velocity given in \eqref{Assumption}:
\begin{align*} 
    q=|q|e^{i\varphi}\in L^\infty(\R;\C),\quad |q|>0, \quad \d_x\varphi\in L^\infty(\R;\R).
\end{align*} 
By the unitary equivalence between $L$ and $\cL$ given in Lemma \ref{lem:L-cL}, it suffices to study the eigenvalue problem for the operator $\cL$, which will be reformulated   into $\lambda$-nonlinear eigenvalue problems in Subsection \ref{sec:cL}.
We will analyze these $\lambda$-nonlinear eigenvalue problems to derive  estimates for the eigenvalues of $L, \cL$ given in Theorem \ref{thm:cL} in Subsection \ref{sec:ev}.

 \subsection{Eigenvalue problem of $\mathcal{L}$}\label{sec:cL}
 We consider the  eigenvalue problem for $\cL$:
 \begin{equation}\label{cL,lambda}
     \mathcal{L}\Psi=\left( \begin{matrix}
-u_- &  i\d_x \\  i\d_x & -u_+
 \end{matrix} \right)\Psi=\lambda\Psi, \hbox{ with }\Psi=\begin{pmatrix}\Psi_1\\ \Psi_2\end{pmatrix}.
 \end{equation}
 We are going to reformulate it in different interesting cases.
We notice the following symmetry
\begin{equation}\label{symmetry}
(u_+, u_-, \lambda, \Psi_1, \Psi_2)\mapsto (-u_-, -u_+,-\lambda, -\Psi_2, \Psi_1)\end{equation}
in this spectral problem $\cL\Psi=\lambda\Psi$, which  corresponds to the symmetry 
$$(q, \lambda, \psi_1, \psi_2)\mapsto  (\bar q, -\lambda, \psi_2, \psi_1)$$
in the spectral problem for the Lax operator $L\psi=\lambda\psi$, with $\psi=\begin{pmatrix}
\psi_1\\ \psi_2
\end{pmatrix}$. 
 
 \subsubsection{Case $u_\pm=\pm 1$} 
 If $\lambda\neq 1$, then the above spectral problem \eqref{cL,lambda} reads simply   
\begin{equation*}\begin{split}
  \left\{\begin{array}{l} (\lambda-1)\Psi_1=i\d_x\Psi_2,
 \\
 -\d_{xx}\Psi_2-(\lambda^2-1)\Psi_2=0. \end{array}\right.
 \end{split}\end{equation*}
That is, the second component $\Psi_2$ solves the spectral problem for the free Schr\"odinger operator $-\d_{xx}$ with the spectral parameter $(\lambda^2-1)$, and the first component $\Psi_1$ is given by $\frac{1}{\lambda-1}i\d_x\Psi_2$.

By a similar argument or by the symmetry property \eqref{symmetry}, if $\lambda\neq-1$, then the  spectral problem  \eqref{cL,lambda} reads   \begin{equation*} 
      \left\{\begin{array}{l} 
  - \d_{xx}\Psi_1 -(\lambda^2-1)\Psi_1=0,
 \\
   (\lambda+1)\Psi_2=i\d_x\Psi_1, \end{array}\right.
 \end{equation*}
i.e. the first component $\Psi_1$ solves the spectral problem for the free Schr\"odinger operator $-\d_{xx}$ with the spectral parameter $(\lambda^2-1)$, and the second component $\Psi_2$ is given by $\frac{1}{\lambda+1}i\d_x\Psi_1$.
 
  \subsubsection{Case $u_-=-1$}
If $\lambda\neq 1$, then the spectral problem \eqref{cL,lambda} reads  
\begin{equation}\label{Dirac:u+}\begin{split}
 \left\{\begin{array}{l} (\lambda-1)\Psi_1=i\d_x\Psi_2,
 \\
 -\d_{xx}\Psi_2-(\lambda-1)(\lambda+u_+)\Psi_2=0, \end{array}\right.
 \end{split}\end{equation}
and it suffices to consider the  $\lambda$-nonlinear eigenvalue problem for $\Psi_2$:
 \begin{equation*}\label{Schrodinger:V+}
   -\d_{xx}\phi-(\lambda-1)(\lambda+u_+)\phi=0,
 \end{equation*} 
with the first component $\Psi_1$   given by $\frac{1}{\lambda-1}i\d_x\phi$.

  Obviously $\lambda=1$ is not an eigenvalue of $\cL$ in this case $u_-=-1$.
 
 \subsubsection{Case $u_+=1$}
Similarly as above or by the symmetry property \eqref{symmetry}, if $\lambda\neq -1$, then the spectral problem \eqref{cL,lambda} reads  \begin{equation}\label{Dirac:u-}\begin{split}
 \left\{\begin{array}{l}  -\d_{xx}\Psi_1-(\lambda+1)(\lambda+u_-)\Psi_1=0,
 \\
 (\lambda+1)\Psi_2=i\d_x\Psi_1, \end{array}\right.
 \end{split}\end{equation} 
 and $\lambda=-1$ is not an eigenvalue of $\cL$ in this case $u_+=1$.

 \subsubsection{General case of $u_\pm\in L^\infty(\R;\R)$ and $\lambda$  such that $\lambda+u_-\neq 0$ on $\R$}By straightforward calculations, the spectral problem    \eqref{cL,lambda} reads as
 \begin{equation}\label{u+,lambda}
      \left\{\begin{array}{l} 
      (\lambda+u_-)\Psi_1=i\d_x\Psi_2,
 \\
- \d_x\Bigl(  \frac{1}{\lambda+u_-}\d_x\Psi_2\Bigr)-(\lambda+u_+)\Psi_2=0. \end{array}\right.
 \end{equation}

 \subsubsection{General case of $u_\pm\in L^\infty(\R;\R)$ and $\lambda$  such that $\lambda+u_+\neq 0$ on $\R$}As above, \eqref{cL,lambda} becomes  
 \begin{equation}\label{u-,lambda}
      \left\{\begin{array}{l} 
  - \d_x\Bigl(  \frac{1}{\lambda+u_+}\d_x\Psi_1\Bigr)-(\lambda+u_-)\Psi_1=0,
 \\
   (\lambda+u_+)\Psi_2=i\d_x\Psi_1. \end{array}\right.
 \end{equation}

To conclude, we have  
\begin{lem}[Reformulation of the eigenvalue problem of $\cL$]\label{lem:cL}
 The eigenvalue problem $\cL\Psi=\lambda\Psi$ reads, 
 \begin{enumerate}
     \item if $\lambda+u_-\neq 0$ on $\R$, as
     \begin{equation}\label{lambda:phi,+}
         - \d_x\Bigl(  \frac{1}{\lambda+u_-}\d_x\phi\Bigr)-(\lambda+u_+)\phi=0,
     \end{equation}
    together with $\Psi=\begin{pmatrix}\Psi_1\\ \Psi_2\end{pmatrix}=\begin{pmatrix} \frac{1}{\lambda+u_-}i\d_x\phi\\ \phi
 \end{pmatrix}$.
 
 \item if $\lambda+u_+\neq 0$ on $\R$, as
     \begin{equation}\label{lambda:phi,-}
         - \d_x\Bigl(  \frac{1}{\lambda+u_+}\d_x\phi\Bigr)-(\lambda+u_-)\phi=0,
     \end{equation}
    together with $\Psi=\begin{pmatrix}\Psi_1\\ \Psi_2\end{pmatrix}=\begin{pmatrix} \phi \\ \frac{1}{\lambda+u_+}i\d_x\phi
 \end{pmatrix}$.
 \end{enumerate}
 
\end{lem}

 \subsection{Proof of Theorem \ref{thm:cL}}\label{sec:ev}
 In this subsection we prove Theorem \ref{thm:cL}.
 We consider first $c=1$, and for notational simplicity we introduce   two real-valued functions
 \begin{equation}\label{Vpm}
     V_\pm=u_\pm\mp 1=\frac12v\pm(|q|-1).
 \end{equation}
From now on we assume that $u_\pm\in L^\infty(\R; \R)$, and hence $V_\pm \in L^\infty(\R;\R)$,
and we consider the eigenvalue problem   of the operator 
$$\cL =  \begin{pmatrix}
1-V_- &  i\d_x \\  i\d_x & -1-V_+
 \end{pmatrix}  : H^1(\R;\C^2)\to L^2(\R;\C^2).$$

We decompose $V_\pm$ into their positive and negative parts respectively $$V_\pm=(V_\pm)^{+}-(V_\pm)^{-},\hbox{ with }(V_\pm)^{+}=\max\{V_\pm, 0\},\,\, (V_\pm)^{-}=\max\{-V_\pm, 0\}.$$
If $\lambda$ is an eigenvalue of $\cL$ such that 
\begin{align*} 
    1-\lambda>\|(V_-)^+\|_{L^\infty},
\end{align*}
then $$-(\lambda+u_-)=1-\lambda-V_-=1-\lambda-(V_-)^++(V_-)^->0 \hbox{ on }\R.$$
By Lemma \ref{lem:cL},  the eigenvalue problem $\cL\Psi=\lambda\Psi$ reads as the $\lambda$-nonlinear eigenvalue problem \eqref{lambda:phi,+}:
\begin{equation*} 
         -  \d_x\Bigl(  \frac{1}{1-\lambda-V_-}\d_x\phi\Bigr)+(1+\lambda+V_+)\phi=0.
     \end{equation*} 
  We test it
  by $\bar\phi$ to derive
 \begin{align*}
     \int_{\R}\Bigl(\frac{1}{1-\lambda-V_-}|\d_x\phi|^2+(1+\lambda+V_+)|\phi|^2\Bigr)\dx=0.
 \end{align*} 
This yields
 $$
  0\geq \int_{\R} (1+\lambda+V_+)|\phi|^2 \dx\geq (1+\lambda -\|(V_+)^-\|_{L^\infty})\|\phi\|_{L^2}^2.
 $$
 We then have shown  for $\lambda$ an eigenvalue of $\cL$,
 \begin{equation*}\label{Case+}
  1-\lambda>\|(V_-)^+\|_{L^\infty}\Longrightarrow 
1+\lambda\leq \|(V_+)^-\|_{L^\infty}.
\end{equation*}
Hence equivalently (or following the same argument)
 \begin{equation*}\label{Case-}
  1+\lambda>\|(V_+)^-\|_{L^\infty}\Longrightarrow 
1-\lambda\leq \|(V_-)^+\|_{L^\infty}.
\end{equation*}
  We conclude  
 the following unconditional statement for the eigenvalues $\lambda$ of $\cL$:
 $$
  1-\lambda\leq \|(u_-+1)^{+}\|_{L^\infty}
    \quad  \hbox{ or }\quad 1+\lambda \leq \|(u_+-1)^{-}\|_{L^\infty}.
 $$

 More generally, $\forall c\in\R$,  we can replace $\pm1$ and $u_\pm\mp1$ by $\pm c$ and $u_\pm\mp c$ respectively in the above arguments.
 This completes the proof of   Theorem \ref{thm:cL}.

 
 \section{Special case of potentials under nonzero boundary conditions}\label{sec:GP}
 \setcounter{equation}{0}
 In this section we will consider the Lax operator $L=\left(\begin{matrix}
i\d_x&-iq \\ i\bar q&-i\d_x
  \end{matrix}\right): H^1(\R;\C^2)\to L^2(\R; \C^2)$ with    bounded potentials $q=q(x): \R\to\C\backslash\{0\}$   satisfying 
 \begin{equation}\label{Assumption:GP}\begin{split}
    &q=|q|e^{i\varphi}\in L^\infty(\R;\C),\quad |q|>0,\quad \d_x\varphi\in L^\infty(\R;\R),
    \\
&   |q(x)|\rightarrow1 \hbox{ and }\d_x\varphi(x)\rightarrow0\hbox{ as }|x|\rightarrow\infty.
\end{split}\end{equation} 
Thus the potentials $u_\pm=\frac12\d_x\varphi\pm|q|$ of its unitarily equivalent operator $\mathcal{L}=\left( \begin{matrix}
-u_- &  i\d_x \\  i\d_x & -u_+
 \end{matrix} \right):H^1(\R;\C^2)\to L^2(\R; \C^2) $  satisfy
\begin{equation*}
    u_\pm\in L^\infty(\R; \R),\quad  u_\pm(x)\rightarrow \pm 1 \hbox{ as }|x|\rightarrow\infty.
\end{equation*}
As in Subsection \ref{sec:ev} we introduce
 \begin{equation}\label{Vpm}
    V_\pm=u_\pm\mp 1\in L^\infty(\R; \R),\hbox{ which satisfy }  V_\pm(x)\rightarrow 0 \hbox{ as }|x|\rightarrow\infty,
 \end{equation}
and we decompose $V_\pm$ into their positive and negative parts respectively $$V_\pm=(V_\pm)^{+}-(V_\pm)^{-},\hbox{ with }(V_\pm)^{+}=\max\{V_\pm, 0\},\,\, (V_\pm)^{-}=\max\{-V_\pm, 0\}.$$ 
Under the smallness assumption \eqref{smallness:upm}: 
\begin{equation}\label{smallness}
    \|(V_+)^-\|_{L^\infty}+\|(V_-)^+\|_{L^\infty}<2,
\end{equation}
by Theorem \ref{thm:cL}, all the eigenvalues $\lambda\in (-1,1)$ of $L, \cL$ satisfy
$$
\hbox{ either }\quad \lambda\in I_-=(-1, -1+\|(V_+)^-\|_{L^\infty}]
\quad \hbox{ or }\quad \lambda \in I_+=[1-\|(V_-)^+\|_{L^\infty}, 1),
$$
where $I_+$ and $I_-$ are disjoint. 
For notational convenience we introduce further  two intervals 
\begin{equation*}\label{nonzero}\begin{split}
    &    J_-:=[-1, 1-\|(V_-)^+\|_{L^\infty}),
\quad    J_+:=(-1+\|(V_+)^-\|_{L^\infty}, 1],
\end{split}\end{equation*}
so that $ I_-\subsetneq  J_-$ and $ I_+\subsetneq J_+$.
The eigenvalues $\lambda\in (-1,1)$ have been discussed in Subsection \ref{Case+}, where we have shown  
\begin{equation}\label{ev,I-}\begin{split}
    &\lambda\in J_-\Rightarrow \lambda+u_-=\lambda-1+V_-<0\hbox{ on }\R,
    \\
    &\lambda \in \mathring J_-=J_-\backslash\{-1\}\hbox{ is an eigenvalue of }\cL\Rightarrow \lambda\in I_-,
\end{split}\end{equation}
and similarly (or equivalently)
\begin{equation}\label{ev,I+}\begin{split}
    &\lambda\in J_+\Rightarrow \lambda+u_+=\lambda+1+V_+>0\hbox{ on }\R,
    \\
    &\lambda \in \mathring J_+ 
    =(-1,1)\backslash I_-\hbox{ is an eigenvalue of }\cL\Rightarrow \lambda\in I_+=(-1,1)\backslash \mathring J_-.
\end{split}\end{equation}

This section is organized as follows.
In Subsection \ref{subs:number} we are going to prove Theorem \ref{thm:lambda}, which characterizes the number of the eigenvalues of $L,\cL$ inside $(-1,1)$.
Some examples showing the (non-)existence of   eigenvalues  will be given in Subsection \ref{subs:example}.  

In the appendix some interesting characterization of the eigenvalues inside $(-1,1)$ in the special case $V_-=0$ or $V_+=0$ will be discussed.

 \subsection{Proof of Theorem \ref{thm:lambda}}\label{subs:number}
In this subsection we are going to consider the eigenvalues of $\cL$ inside $(-1,1)$, under the assumptions \eqref{Vpm}-\eqref{smallness}. 
By Lemma \ref{lem:cL} and \eqref{ev,I-}-\eqref{ev,I+}, it suffices to consider  the eigenvalues
$ \lambda\in \mathring J_- $ of the eigenvalue problem \eqref{lambda:phi,+}:
\begin{equation}\label{ev,lambda,phi,+}
         - \d_x\Bigl(  \frac{1}{1-V_--\lambda}\d_x\phi\Bigr)+(1+V_++\lambda)\phi=0,
         \quad x\in\R,
     \end{equation}
     and the eigenvalues $\lambda\in \mathring J_+$ of the eigenvalue problem \eqref{lambda:phi,-}:  \begin{equation}\label{ev,lambda,phi,-}
         - \d_x\Bigl(  \frac{1}{1+V_++\lambda}\d_x\phi\Bigr)+(1-V_--\lambda)\phi=0, \quad x\in\R.
     \end{equation} 
Our goal is to prove Theorem \ref{thm:lambda}, that is,   the number of eigenvalues inside $\mathring{J}_-$ of \eqref{ev,lambda,phi,+} (resp. $\mathring{J}_+$ of \eqref{ev,lambda,phi,-}) is the same as  
the number of negative eigenvalues of the Sturm-Liouville eigenvalue problem \eqref{ev,mu,intro} (resp. \eqref{ev,nu,intro}) reading as follows:
\begin{equation}\label{ev,mu}
    -\d_x\Bigl( \frac{1}{2-V_-}\d_x\psi\Bigr)+V_+\psi=\mu \psi,\quad x\in\R,
\end{equation}
resp. 
\begin{equation}\label{ev,nu}
     -\d_x\Bigl( \frac{1}{2+V_+}\d_x\psi\Bigr)-V_-\psi=\nu \psi,\quad x\in\R.
\end{equation}
By compact perturbation arguments, we deduce from the decay property of $V_\pm$ in \eqref{Vpm} that the essential spectra of both \eqref{ev,mu} and \eqref{ev,nu} are contained in $[0,\infty)$.

We  first do some preparations for the proof of Theorem \ref{thm:lambda}. 
We consider, for $ \lambda\in J_-=[-1, 1-\|(V_-)^+\|_{L^\infty})$, the   eigenvalue problem (noticing $1-V_--\lambda>0$ on $\R$ in this case)
\begin{equation}\label{ev,lambda,mu}
         - \d_x\Bigl(  \frac{1}{1-V_--\lambda}\,\d_x\psi\Bigr)+(1+V_++\lambda)\psi=\mu \psi,
         \quad x\in\R,
     \end{equation}
     and for $\lambda\in  J_+=(-1+\|(V_+)^-\|_{L^\infty},1]$,   the eigenvalue problem  (noticing $1+V_++\lambda>0$ on $\R$ in this case)
     \begin{equation}\label{ev,lambda,nu}
         - \d_x\Bigl(  \frac{1}{1+V_++\lambda}\,\d_x\psi\Bigr)+(1-V_--\lambda)\psi=\nu\psi,
         \quad x\in\R.
     \end{equation} 
     Here, $\lambda$ is regarded as a given parameter, and $\mu$ (resp. $\nu$) is the eigenvalue parameter.
     In particular, 
     \begin{align*}
        & \lambda=-1\Rightarrow \hbox{ the two eigenvalue problems \eqref{ev,mu}, \eqref{ev,lambda,mu} are the same},\\
        & \lambda=1\Rightarrow \hbox{ the two eigenvalue problems \eqref{ev,nu}, \eqref{ev,lambda,nu} are the same}.
     \end{align*}

     We define the Rayleigh quotients for $\psi\in H^1(\R)\backslash\{0\}$
     \begin{equation}\label{R,lambda}
         R(\lambda, \psi):=\frac{\displaystyle\int_{\R} \Bigl[ \frac{1}{1-V_--\lambda}|\d_x\psi|^2+(1+V_++\lambda)|\psi|^2 \Bigr]\dx}{\int_{\R} |\psi|^2\dx}
     \end{equation}
     for the eigenvalue problem \eqref{ev,lambda,mu}, and 
     \begin{equation}\label{S,lambda}
         S(\lambda, \psi):=\frac{\displaystyle\int_{\R} \Bigl[ \frac{1}{1+V_++\lambda}|\d_x\psi|^2+(1-V_--\lambda)|\psi|^2 \Bigr]\dx}{\int_{\R} |\psi|^2\dx}
     \end{equation}
     for the eigenvalue problem \eqref{ev,lambda,nu},
   and respectively the Rayleigh extremal values
     \begin{equation}\label{muj,lambda}
         \mu_j(\lambda):=  \mathop{\inf}\limits_{\substack{ U\subset H^1(\R)\hbox{ subspace}, \\  \dim(U)=j}}\,\max_{\psi\in U\backslash\{0\}} R(\lambda,\psi),\quad j\in\N,
     \end{equation}
     \begin{equation}\label{nuj,lambda}
         \nu_j(\lambda):=  \inf_{\substack{ U\subset H^1(\R)\hbox{ subspace}, \\  \dim(U)=j}}\,\max_{\psi\in U\backslash\{0\}} S(\lambda,\psi),\quad j\in\N.
     \end{equation}

We first observe that for all $\lambda_0\in (-1+\|(V_+)^-\|_{L^\infty}, 1-\|(V_-)^+\|_{L^\infty})=J_-\cap J_+$, we have 
\begin{align}\label{muj,lambda0}
   &  \mu_j(\lambda_0)\geq 1-\|(V_+)^-\|_{L^\infty}+\lambda_0>0, \quad\forall j\in\N, \quad  \hbox{(by  \eqref{R,lambda} and \eqref{muj,lambda})},
\end{align} 
\begin{align}\label{nuj,lambda0}
   \nu_j(\lambda_0)\geq 1-\|(V_-)^+\|_{L^\infty}-\lambda_0>0,\quad\forall j\in\N,\quad  \hbox{(by  \eqref{S,lambda}  and \eqref{nuj,lambda})}.
\end{align} 
    By the decay property of $V_\pm(x)\rightarrow0$ as $|x|\rightarrow\infty$, we  observe that the essential spectrum of \eqref{ev,lambda,mu} (resp. \eqref{ev,lambda,nu}) is contained in $[1+\lambda,\infty)$ (resp. $[1-\lambda,\infty)$).
    Thus, cf.   \cite[Theorem 10.33]{NPW}, one has
\begin{equation}\label{ev,muj,nuj} \begin{split}
 &\hbox{If }\mu_j(\lambda)<1+\lambda, \hbox{then } \mu_j(\lambda) \hbox{ is the }j\hbox{-th eigenvalue of \eqref{ev,lambda,mu}},
 \\
 &\hbox{If }   \nu_j(\lambda)<1-\lambda, 
 \hbox{then }   \nu_j(\lambda)  \hbox{ is the }j\hbox{-th eigenvalue of  \eqref{ev,lambda,nu}}.
 \end{split}\end{equation}   
In particular,  we observe that
\begin{equation}\label{ev,lambdaj}\begin{split}
&\hbox{If }\mu_j(\lambda_j)=0\hbox{ for some }\lambda_j\in \mathring J_-,\,\, j\in\N, \hbox{ then } \lambda_j   \hbox{ is an  eigenvalue of \eqref{ev,lambda,phi,+}};\\
&
  \hbox{If }  \nu_j(\lambda_j)=0 \hbox{ for some }\lambda_j\in\mathring J_+,\,\, j\in\N, 
  \hbox{ then } \lambda_j  \hbox{ is an   eigenvalue of \eqref{ev,lambda,phi,-}}.
\end{split}\end{equation}  
\begin{lem}[Properties of $\mu_j, \nu_j$]\label{lem:mu,nu}
For $j\in\N$, the mappings
\begin{align*}
   & \mu_j=\mu_j(\lambda): J_-=[-1,1-\|(V_-)^+\|_{L^\infty})\to \R 
   \\
&    \hbox{and } \nu_j=\nu_j(\lambda): J_+=(-1+\|(V_+)^-\|_{L^\infty}, 1]\to \R 
\end{align*}
are continuous, and $\mu_j$ (resp. $\nu_j$) is strictly increasing (resp. decreasing) in $\lambda$. 
\end{lem}
\begin{proof}   
Let $ \lambda\in  J_-$, and we are going to show the continuity of the mapping $\mu_j$ at $\lambda$.
For any $\tilde\lambda\in J_-$ and  any fixed $\psi\in H^1(\R)\backslash\{0\}$, we estimate
\begin{align*}
    & |R(\lambda,\psi)-R(\tilde\lambda,\psi)|
      \leq \frac{|\lambda-\tilde\lambda|}{\int_{\R} |\psi|^2\dx}\int_{\R} \Bigl[ \frac{1}{(1-V_--\lambda)(1-V_--\tilde\lambda)}|\d_x\psi|^2+ |\psi|^2 \Bigr]\dx 
   \\
   &\leq \frac{|\lambda-\tilde\lambda|}{\int_{\R} |\psi|^2\dx}
  \Bigl(  \frac{1}{1-\|(V_-)^+\|_{L^\infty}-\tilde\lambda}  \cdot \int_{\R}  \Bigl[  \frac{1}{ (1-V_--\lambda)}|\d_x\psi|^2+ (1+V_++\lambda)|\psi|^2 \Bigr]\dx
   \\
   &\qquad \qquad \qquad 
   +\Bigl[ \frac{-1+\|(V_+)^-\|_{L^\infty}-\lambda}{1-\|(V_-)^+\|_{L^\infty}-\tilde\lambda}+1 \Bigr]
   \cdot \int_{\R} |\psi|^2\dx\Bigr)
   \\
   &= |\lambda-\tilde\lambda|\Bigl(   \frac{1}{1-\|(V_-)^+\|_{L^\infty}-\tilde\lambda}   R(\lambda,\psi)
   +\Bigl[ \frac{-1+\|(V_+)^-\|_{L^\infty}-\lambda}{1-\|(V_-)^+\|_{L^\infty}-\tilde\lambda}+1 \Bigr]\Bigr).
\end{align*}
If $\tilde\lambda\in J_-$ satisfies  $|\lambda-\tilde\lambda|<\frac12 (1-\|(V_-)^+\|_{L^\infty}-\lambda)$ and hence 
$    \frac{2}{3(1-\|(V_-)^+\|_{L^\infty}-\lambda)}<\frac{1}{1-\|(V_-)^+\|_{L^\infty}-\tilde\lambda} <\frac{2}{ 1-\|(V_-)^+\|_{L^\infty}-\lambda},
$
then there exist two positive constants $C_1, C_2$ which are independent of $\tilde\lambda$ and $\psi$ (depending only on $\lambda, \|(V_+)^-\|_{L^\infty}, \|(V_-)^+\|_{L^\infty}$), such that
\begin{align*}
    & |R(\lambda,\psi)-R(\tilde\lambda,\psi)|
       \leq  |\lambda-\tilde\lambda|\Bigl(  C_1   R(\lambda,\psi)
   +C_2\Bigr).
\end{align*}
Thus for   $\tilde\lambda\in J_-$ satisfying 
$ |\lambda-\tilde\lambda|< \min\Bigl\{ \frac12 (1-\|(V_-)^+\|_{L^\infty}-\lambda), \, \frac{1}{C_1}\Bigr\},
$
we have
$$
(1-|\lambda-\tilde\lambda|C_1) R(\lambda,\psi)-|\lambda-\tilde\lambda|C_2\leq R(\tilde\lambda,\psi)
\leq (1+|\lambda-\tilde\lambda|C_1) R(\lambda,\psi)+|\lambda-\tilde\lambda|C_2,
$$
which, together with \eqref{muj,lambda}, implies for each $j\in\N$,
\begin{align*}
    (1-|\lambda-\tilde\lambda|C_1) \mu_j(\lambda)-|\lambda-\tilde\lambda|C_2\leq \mu_j(\tilde\lambda)
\leq (1+|\lambda-\tilde\lambda|C_1) \mu_j(\lambda)+|\lambda-\tilde\lambda|C_2.
\end{align*}
This implies the continuity of the mapping $\mu_j$ at $\lambda$.

By \eqref{R,lambda} and \eqref{muj,lambda},
   \begin{equation*}\label{muj,lambda,0}
         \mu_j(\lambda) = (1+\lambda)+ \mathop{\inf}\limits_{\substack{ U\subset H^1(\R)\hbox{ subspace}, \\  \dim(U)=j}}\,\max_{\psi\in U\backslash\{0\}} \frac{\displaystyle\int_{\R} \Bigl[ \frac{1}{1-V_--\lambda}|\d_x\psi|^2+ V_+ |\psi|^2 \Bigr]\dx}{\int_{\R} |\psi|^2\dx}.
     \end{equation*}
     Hence  $\mu_j(\lambda)$ is strictly increasing in $\lambda\in J_-$.



\smallbreak

Analogously, we can argue exactly as above,  with $ \lambda, \tilde\lambda,   V_+, V_-,  R, \mu$ replaced by $-\lambda, -\tilde\lambda,   -V_-,- V_+,  S, \nu$ respectively, to derive the continuity and the monotonicity  of the mapping $\nu_j$.
\end{proof}

We are now ready to prove Theorem \ref{thm:lambda}.
\begin{proof}[Proof of Theorem \ref{thm:lambda}]
 Let $m\in \N_0\cup\{\infty\}$ denote the number of negative eigenvalues $\mu$ of the eigenvalue problem \eqref{ev,mu}, that is, the eigenvalue problem \eqref{ev,lambda,mu} with $\lambda=-1$.
 We take these  negative eigenvalues
\begin{equation}\label{muj}
 \mu_j(-1)<0,\quad  j=1,\cdots,m \hbox{ if }m\in\N \,\hbox{ or }\,j\in\N\hbox{ if }m=\infty.
 \end{equation}
 By Lemma \ref{lem:mu,nu} and \eqref{muj,lambda0}, for each such $j$, there exists a unique $\lambda_j\in (-1,\lambda_0)\subset \mathring J_-$ such that 
 \begin{equation}\label{muj,lambdaj}
     \mu_j(\lambda_j)=0.
 \end{equation}
 By  \eqref{ev,lambdaj}, $\lambda_j\in \mathring J_-$ is an eigenvalue of \eqref{ev,lambda,phi,+}, and hence of the operator $\cL$ (by Lemma \ref{lem:cL}), which implies $\lambda_j\in I_-$ (by \eqref{ev,I-}). 
Furthermore, if for some $j$ and $k$, the functions $\mu_j,\cdots, \mu_{j+k}$ have the same zero with $\lambda_j=\cdots=\lambda_{j+k}$, 
that is, $0$ is a $(k+1)$-fold eigenvalue of \eqref{ev,lambda,mu} for this value of $\lambda=\lambda_j$, then  the associated eigenfunctions of the eigenvalue problem \eqref{ev,mu} are linearly independent. 
Thus, counting by multiplicity, we obtain at least $m$ eigenvalues inside $I_-$ of the operator $\cL$.

Vice versa, when $\lambda\in I_-\subset \mathring J_-$ is an eigenvalue of $\cL$, then $\lambda$ is an eigenvalue of \eqref{ev,lambda,phi,+}, and hence $\mu=0$ is an eigenvalue of \eqref{ev,lambda,mu}.  Furthermore, the eigenvalue $\mu=0$ is below  the essential spectrum $[1+\lambda,\infty)$.
Thus,  by \eqref{ev,muj,nuj}, there exists some $j\in\N$ such that $\mu_j(\lambda)=0$.
By the strict monotonicity in Lemma \ref{lem:mu,nu}, we have $\mu_j(-1)<0$.
By \eqref{ev,muj,nuj}, $\mu_j(-1)$ is a negative eigenvalue of \eqref{ev,mu}, which implies $j\in \{1,\cdots,m\}$ if $m\in\N$.
Since the zero $\lambda_j$ of $\mu_j$ is unique, we obtain $\lambda=\lambda_j$.
If $\lambda$ has multiplicity $(k+1)$ as an eigenvalue of $\cL$, then $\mu=0$ has multiplicity $(k+1)$ as an eigenvalue of \eqref{ev,lambda,mu}, and hence $\mu_j(\lambda)=\cdots=\mu_{j+k}(\lambda)=0$ for some $j\in\N$.
As before we find $\lambda=\lambda_j=\cdots=\lambda_{j+k}$.
This implies, counting by multiplicity, that the operator $\cL$ has at most $m$ eigenvalues in $I_-$. 
This completes the proof of (1) in Theorem \ref{thm:lambda} if $m\in\N$.
The above arguments show that it is also true if $m=\infty$.

 \smallbreak
 
 Part (2) in Theorem \ref{thm:lambda} follows exactly as above, if we replace $\lambda, m, \mu, I_-, J_-$ by $-\lambda, l, \nu, I_+, J_+$ respectively. 
\end{proof}

\subsection{Examples}\label{subs:example}
In this subsection we will give some examples, where the   Lax operator $L$ in \eqref{L}, or equivalently its unitary equivalence $\cL$ in \eqref{cL,intro},  may or may not possess eigenvalues inside the spectral gap $(-1,1)$.

In the following examples we will always assume the nonzero boundary condition at infinity \eqref{BC:q,varraphi} for the nowhere vanishing bounded potential $q$,   so that
\begin{align}\label{BC,Vpm}
     V_\pm\in L^\infty(\R; \R), \quad   V_\pm(x)\rightarrow 0 \hbox{ as }|x|\rightarrow\infty, 
\end{align}
where  $V_\pm$ are related to $q$ as follows:
\begin{align*}
    q=|q|e^{i\varphi},\quad u_\pm=\frac12(\partial_x\varphi)\pm |q|,\quad V_\pm=u_\pm\mp1.
\end{align*}
In order to show the existence of eigenvalues in some examples we will assume further the smallness assumption \eqref{smallness:upm}:
\begin{align} \label{Small,Vpm}
    \|(V_+)^-\|_{L^\infty}+\|(V_-)^+\|_{L^\infty}<2.
\end{align}

\begin{ex}[No eigenvalues]
For the   potentials of the following type
$$
 q=|q|e^{i\varphi}\in L^\infty(\R),\quad |q|-1\geq 0, \quad \varphi(x)=\int^x_{x_0} 2(|q|-1) \dy,
 $$
for some real number $x_0\in\R$, we have $\frac12\partial_x\varphi=|q|-1$ and 
 \begin{align*}
    u_+=   2|q|-1\geq 1,
    \quad u_-=-1.
\end{align*} 
Therefore by Theorem \ref{thm:cL} (with $c=1$) there are no eigenvalues of the Lax operator $L$ inside $(-1,1)$. 
 If we assume further decay condition \eqref{potential,decay} at infinity: $(1+x^4)(q-e^{i\theta_\pm})\in L^1(\R^\pm)$ for some $\theta_\pm\in\R$,
 then by \cite{DPVV} there are only finitely many discrete eigenvalues which are  located in $(-1,1)$, and hence the spectrum of $L$ consists only of essential spectrum $(-\infty,-1]\cup[1,\infty)$.
\end{ex}

\begin{ex}[Existence of eigenvalues]\label{ex:ev}
\begin{enumerate}[(i)]
    \item Let  the potentials  $V_\pm$ satisfy \eqref{BC,Vpm}-\eqref{Small,Vpm}  and the following assumption:
 \begin{equation}\label{V+,R}
     \int_{-R}^R V_+(x)\dx<0\hbox{ and }V_+\leq 0\hbox{ outside }[-R, R],
 \end{equation} 
 for some $R>0$.

 We are going to show  the existence of negative eigenvalues for \eqref{ev,mu,intro} (i.e. \eqref{ev,mu}): 
\begin{equation}\label{ev,mu,app}
-\d_x\Bigl( \frac{1}{2-V_-}\d_x\psi\Bigr)+V_+\psi=\mu \psi
\end{equation}
by the Rayleigh-Ritz method, cf. \cite[Theorem 10.23]{NPW}.
More precisely, we can try to prove that the eigenvalue problem \eqref{ev,mu,app} has at least $n$ negative eigenvalues (with $n\in \N$ chosen as we like): We have to select linearly independent $\psi_1,\cdots, \psi_n\in H^1(\R)$, compute the numbers
    $$
    A_{jk}:=\int_{\R} \Bigl( \frac{1}{2-V_-} \d_x\psi_j\d_x\bar \psi_k+V_+\psi_j\bar\psi_k\Bigr) \dx,\quad j,k=1,\cdots,n,
    $$  and show that the matrix $A=(A_{jk})$ has $n$ negative eigenvalues (by explicit calculation or by computer assistance).
    If $A$ turns out to have less than $n$ negative eigenvalues, $n$ was chosen too large, and we can retry with smaller $n$. 
    In the following we will just choose $n=1$ to show the existence of eigenvalues.

We take $\tilde\psi\in C_c^\infty(\R)$ such that $\tilde\psi=1$ on $(-1,1)$ and $\tilde\psi=0$ outside $(-2,2)$.
Let $\psi(x)=\tilde\psi(\frac{x}{k})$, $x\in \R$, where $k\geq R$ is chosen large enough such that
 (by virtue of the assumptions \eqref{Small,Vpm} and \eqref{V+,R})
 \begin{align*}
    A_{11}&=\int_{\R} \Bigl( \frac{1}{2-V_-} |\d_x\psi|^2 +V_+|\psi|^2\Bigr) \dx 
    \\
&     \leq \frac{1}{2-\|(V_-)^+\|_{L^\infty}}\Bigl(\frac{1}{k} \|\d_x\tilde\psi\|_{L^2}^2\Bigr)
      +\Bigl( \int_{-R}^R V_+ \dx\Bigr)  <0.
 \end{align*}
 Hence there exists at least one negative  eigenvalue $\mu$ of \eqref{ev,mu,app} (i.e. \eqref{ev,mu,intro}), and thus there exists at least one eigenvalue $\lambda\in I_-=(-1,-1+\|(V_+)^{-}\|_{L^\infty}]$ of the operator  $L,\cL$   (by Theorem \ref{thm:lambda}). 
 
 \item  Analogously, if the potentials  $V_\pm$ satisfy \eqref{BC,Vpm}-\eqref{Small,Vpm} and the   assumption:
 \begin{equation*} 
     \int_{-R}^R V_-(x)\dx>0\hbox{ and }V_-\geq 0\hbox{ outside }[-R, R],
 \end{equation*} 
 for some $R>0$, then there exists at least one negative eigenvalue $\nu$ of the eigenvalue problem \eqref{ev,nu,intro} (i.e. \eqref{ev,nu}): 
\begin{equation}\label{ev,nu,app}
-\d_x\Bigl( \frac{1}{2+V_+}\d_x\psi\Bigr)-V_-\psi=\nu \psi,
\end{equation}
and hence the operator $L$ has at least one eigenvalue $\lambda\in I_+=(1-\|(V_-)^+\|_{L^\infty}, 1]$. 
\end{enumerate} 
\end{ex}

\begin{ex}[Potentials with piecewise-constant amplitude and phase velocity]\label{ex:2} 
We take the potentials of the following form
\begin{equation}\label{potential,linear} 
    q(x)=\left\{\begin{array}{cc}
       e^{i\theta}  & x< a,  \\
       Ae^{i(\theta+v(x-a))}  & a\leq x\leq b,\\
       e^{i(\theta+v(b-a))} & x> b,
    \end{array}\right.
 \end{equation}
 where  $(a,b)$ is an arbitrary nonempty finite interval on $\R$, and $\theta\in [0,2\pi)$, $A>0$ and $v\in\R$ are  real numbers.
 
If $A\geq 1$ and $v\in \R$ such that
\begin{align}
    \frac12|v|\leq A-1, 
\end{align}
then $u_+\geq 1$ and $u_-\leq -1$, and hence by Theorem \ref{thm:cL} (with $c=1$)  the operator $L$ has no eigenvalues inside $(-1,1)$, and thus no eigenvalues by \cite{DPVV} since obviously $q$ in \eqref{potential,linear} satisfies \eqref{potential,decay} with   $\theta_-=\theta$ and $\theta_+=\theta+v(b-a)$.

In the following we will take  
\begin{align}
    A\in (0,1) \hbox{ and }\frac12|v|<A+1,
\end{align}
and show the existence of exactly  one or two eigenvalue(s) of $L$
in different cases:
\begin{enumerate}[(I)]
    \item Exactly two simple eigenvalues, located in $(-1, -\frac12v-A]$ and $[-\frac12v+A,1)$ respectively, if
    $$A\in (0,1)\hbox{ and }\frac12|v|<1-A.$$

In this case  $V_\pm$ read as the following piecewise constant functions
\begin{equation}\label{Vpm,piece}
V_+(x)=\left\{\begin{array}{ll}-\varepsilon& \hbox{ on }[a,b],\\ 0&\hbox{ otherwise,} \end{array}\right. \quad 
 V_-(x)=\left\{\begin{array}{ll} \delta& \hbox{ on }[a,b],\\ 0&\hbox{ otherwise,} \end{array}\right.\end{equation}
 where $\varepsilon=-\frac12v-A+1$, $\delta=\frac12v-A+1$ are two positive constants located inside $(0,2-2A)$, such that the smallness condition \eqref{Small,Vpm} holds:
 $$\|(V_+)^-\|_{L^\infty}+\|(V_-)^+\|_{L^\infty}=\varepsilon+\delta=2-2A<2.$$  
We conclude from Example \ref{ex:ev} that in this case there are at least two eigenvalues of the Lax operator $L$, located in $(-1, -1+\varepsilon]$ and $[1-\delta,1)$ respectively. We are going to show that there are indeed exactly two eigenvalues.
 
 We consider the eigenvalue problem \eqref{ev,mu,intro} (i.e. \eqref{ev,mu,app} above).
An integration by parts shows that any negative eigenvalue $\mu$  lies in $(-\varepsilon, 0)$.
 More precisely, \eqref{ev,mu,app} reads as   
\begin{align*}
&\d_{xx}\psi=k^2 \psi\hbox{ outside }[a,b],
\\
    &-\d_{xx}\psi=\kappa^2 \psi\hbox{ on }[a,b], 
\end{align*}
where the two positive constants $k,\kappa$ read as $k=\sqrt{-2\mu}<  \sqrt{2\varepsilon}<2$ and  $\kappa=\sqrt{(2-\delta)(\mu+\varepsilon)}<\sqrt{\varepsilon(2-\delta)}$.
We then search for the non-trivial solution of the following form which are continuous  at $a$ and $b$:
\begin{align*}
    \psi(x)=\left\{\begin{array}{cc}
        \frac{ce^{ia\kappa}+de^{-ia\kappa}}{e^{ak}}e^{kx}& \hbox{ for }x\leq a,  \\
         ce^{i\kappa x}+de^{-i\kappa x}& \hbox{ for }x\in [a,b],\\
         \frac{ce^{ib\kappa}+de^{-ib\kappa}}{e^{-bk}}e^{-kx}&\hbox{ for }x\geq b.
    \end{array}\right.
\end{align*}
The continuity conditions of $\frac{1}{2-V_-}\d_x\psi$ at $a$ and $b$ read then as
\begin{align*}
    \begin{pmatrix}
    e^{ia\kappa }(\frac{1}{2-\delta}i\kappa-\frac12k) & e^{-ia\kappa}(-\frac{1}{2-\delta}i\kappa-\frac12k)
    \\
    (\frac1{2-\delta}i\kappa+\frac12k)e^{ib\kappa}& (-\frac{1}{2-\delta}i\kappa+\frac12k)e^{-ib\kappa}
    \end{pmatrix}
    \begin{pmatrix}c\\d\end{pmatrix}
    =\begin{pmatrix}
    0\\0
    \end{pmatrix}.
\end{align*}
In order to have a non-trivial solution, the determinant of the matrix on the lefthand side should vanish:
\begin{align*}
  &0= \cot((b-a)\kappa)- \frac{1}{2\sqrt{2(2-\delta)}}
  f(\kappa),
  \\
&  \hbox{ where }f(\kappa):=\frac{(4-\delta)\kappa^2- \varepsilon (2-\delta)^2}{\kappa\sqrt{ \varepsilon(2-\delta)-\kappa^2} },
\end{align*} 
for $\kappa\in (0,\sqrt{\varepsilon(2-\delta)})$.
It is straightforward to check that the function $f(\kappa)$
has a strictly positive derivative $\frac{  \varepsilon(2-\delta) (\delta \kappa^2 +\varepsilon(2-\delta)^2 )  }{ \kappa^2\sqrt{ \varepsilon(2-\delta)-\kappa^2} ^3}>0$  
  for $\kappa\in (0,\sqrt{\varepsilon(2-\delta)})$, and $f(\kappa)$ tends to $-\infty$ as $\kappa\rightarrow0_+$ while to $+\infty$ as $\kappa\rightarrow(\sqrt{\varepsilon(2-\delta)})_-$.
Hence there exists a unique $\kappa\in (0,\sqrt{\varepsilon(2-\delta)})$ such that the above equality holds, and hence the eigenvalue problem \eqref{ev,mu,intro} has a single negative eigenvalue $\mu=-(\varepsilon-\frac{\kappa^2}{2-\delta})$.
Thus the operator $L$ has a single eigenvalue  inside $(-1,-1+\varepsilon]$.

Replacing $V_\pm$ by $- V_\mp$, that is, exchanging $\varepsilon$ and $\delta$, the above argument shows that  the eigenvalue problem \eqref{ev,nu,intro} (i.e. \eqref{ev,nu,app} above) has a unique negative eigenvalue, and thus the operator $L$ has a single eigenvalue inside $[1-\delta,1)$.

To conclude, if $A\in (0,1)$ and $\frac12|v|< 1-A$,  the spectra of the operator $L$ associated to the   potential \eqref{potential,linear}   consists of essential spectrum $(-\infty,-1]\cup[1,\infty)$ and two simple eigenvalues, located in $(-1,-\frac12v-A]$ and $[-\frac12v+A,1)$, respectively.

 \item Exactly one simple eigenvalue, located in $(-1,-\frac12v-A]$, if   $$-1-A<\frac12v\leq A-1<0.$$  
 
 In this case, $V_\pm$ are given in \eqref{Vpm,piece}, with the parameter $\varepsilon \in [2-2A, 2)\subset (0,2)$ and $\delta\in (-2A,0]$.
 We analyze the eigenvalue problem \eqref{ev,mu,intro} as in the above case (I), and derive a single negative eigenvalue $\mu=-(\varepsilon-\frac{\kappa^2}{2-\delta})$.
Since the eigenvalue problem \eqref{ev,nu,intro} has no negative eigenvalues for $V_-\leq 0$, the operator $L$ has exactly one eigenvalue located in $(-1,-1+\varepsilon]=(-1,-\frac12v-A]$.
 
 \item Exactly one simple eigenvalue, located in   $[-\frac12v+A, 1)$ if 
 $$0< 1-A\leq \frac12v < 1+A.$$
 
It follows by   exchanging the values of $\varepsilon$ and $\delta$ in the above case (II),  that the eigenvalue problem \eqref{ev,nu,intro} has a single negative eigenvalues and \eqref{ev,mu,intro} has no negative eigenvalues, and hence $L$ has exactly one eigenvalue located in $[1-\delta, 1)=[-\frac12v+A,1)$.


\end{enumerate}
\end{ex}

\begin{ex}[Piecewise-constant case, continued]\label{ex:disjoint}
We take the potentials of the following form
\begin{equation}\label{potential,linear,c} 
    q(x)=\left\{\begin{array}{cc}
       e^{i\theta}  & x< -2,  \\
       (1-\frac12\delta)e^{i(\theta+\delta(x+2))} & x\in [-2,-1],\\
       e^{i(\theta+\delta)} & x\in (-1,1),\\
       (1-\frac12\varepsilon)e^{i(\theta+\delta -\varepsilon(x-1))}  & x\in [1,2],\\
       e^{i(\theta+\delta-\varepsilon)} & x>2,
    \end{array}\right.
 \end{equation}
 where    $\theta\in [0,2\pi)$ and  $0< \varepsilon, \delta$ such that $\varepsilon+\delta<2$. 
 Then the corresponding $V_\pm$ are supported on disjoint intervals as follows  
 $$V_+(x)=\left\{\begin{array}{ll}-\varepsilon& \hbox{ on }[1,2],\\ 0&\hbox{ otherwise,} \end{array}\right. \quad 
 V_-(x)=\left\{\begin{array}{ll} \delta& \hbox{ on }[-2,-1],\\ 0&\hbox{ otherwise.} \end{array}\right.$$
 
 We remark that in view of Example \ref{ex:2}, it is easy to see that the operator $L$ has no eigenvalues if $\varepsilon=\delta=0$, and has exactly one eigenvalue if $\varepsilon=0$, $\delta\in (0,2)$ or $\varepsilon\in (0,2)$, $\delta=0$.
 In the following we follow exactly the arguments   in Example \ref{ex:2} to show  briefly that $L$ has exactly   two eigenvalues if $\varepsilon, \delta> 0$ and $\varepsilon+\delta<2$. 
     
   We can rewrite the eigenvalue problem \eqref{ev,mu} with $\mu\in (-\varepsilon,0)$ as
\begin{align*}
    &\d_{xx}\psi=k^2 \psi\hbox{ on }(-\infty,-2)\cup (-1,1)\cup (2,\infty)
    \\
    &-\d_{xx}\psi=\kappa_1^2  \psi\hbox{ on }[-2,-1],
    \\
    &-\d_{xx}\psi=\kappa_2^2\psi\hbox{ on }[1,2],
\end{align*}
where $k=\sqrt{-2\mu}\in (0,\sqrt{2\varepsilon})$, $\kappa_1=i\sqrt{-\mu(2-\delta)}=i\sqrt{1-\frac\delta2}k\in i\R^+$, and $\kappa_2=\sqrt{2(\mu+\varepsilon)} =\sqrt{2\varepsilon-k^2}\in (0,\sqrt{2\varepsilon})$.
And we  search for a non-trivial solution $\psi\in H^1(\R)$, such that $\psi$ itself as well as $\frac{1}{2-V_-}\d_x\psi$ are continuous    at $\pm 1$ and $\pm 2$.
By a long but straightforward calculation, the fact that the  determinant of the $4\times 4$ matrix generated  by  these continuity conditions vanishes is equivalent to   the following equality  
\begin{align*}
 0=g(k;\varepsilon,\delta),
 \hbox{ where }g(k;\varepsilon,\delta)&=k^2-\varepsilon+k\sqrt{2\varepsilon-k^2}\cot(\sqrt{2\varepsilon-k^2})\\  
&  +\frac{\varepsilon\delta e^{-4k}}{4-\delta+2\sqrt{2(2-\delta)}\coth(\sqrt{1-\frac\delta2}k)}.
\end{align*} 
We claim that there exists a unique zero $k\in (0,\sqrt{2\varepsilon})$ of $g(k;\varepsilon, \delta)$.
Indeed, the existence of a zero $k\in (0,\sqrt{\varepsilon})$ follows immediately from  the fact that $g(0_+)<0$, and $g>0$ on $[\sqrt{\varepsilon}, \sqrt{2\varepsilon})$ (since $\sqrt{2\varepsilon-k^2}\leq \sqrt{\varepsilon}<\sqrt{2}<\frac{\pi}{2}$ for $k\in [\sqrt{\varepsilon}, \sqrt{2\varepsilon})$).
The uniqueness is however non-trivial,
and we consider the cases $\varepsilon\leq 0.04$ and $\varepsilon\geq 0.04$ separately:
\begin{itemize}
    \item Case $\varepsilon\leq 0.04$.
    In this case one may check that $g'(k)>0$ on $(0,\sqrt{\varepsilon})$ and hence there exists a unique zero $k\in (0,\sqrt{\varepsilon})$.
    \item Case $\varepsilon\geq 0.04$. In this case we use a computer-assisted proof to show the uniqueness.
    More precisely, we cover the parameter range $\varepsilon\geq 0.04$, $\delta\geq 0$, $\varepsilon+\delta\leq 2$ by (finitely many) squares with side-length $0.03$, and on each square $S=[\underline{\varepsilon}, \overline{\varepsilon}] \times [\underline{\delta}, \overline{\delta}]$ we evaluate $g(I_j;S)$, $j=1,\cdots,N$ using interval arithmetic \cite{Rump}, where   $\{I_j\}_{j=1}^N$ is a subdivision of  $[0,\sqrt{\overline{\varepsilon}}]$.
    For all squares $S$, the interval evaluation showed that there exists a unique index $j_0\in \{1,\cdots,N\}$ such that $g(\cup_{j\leq j_0-1}I_j; S)\subset (-\infty, 0)$, $g(\cup_{j\geq j_0+1}I_j; S)\subset (0,\infty)$ and $\frac{\partial g}{\partial k}(I_{j_0}; S)\subset (0,\infty)$.
    This proves (rigorously) the uniqueness of the zero $k\in (0,\sqrt{\varepsilon})$ of $g(k;\varepsilon,\delta)$, for all $(\varepsilon, \delta)\in S$. 
    The covering property of the squares $S$ gives the desired result.
     Note that we have to exclude a ``small" parameter range near $\varepsilon = 0$ 
since the singularity of $g$ at $\varepsilon = 0$ is difficult to capture by interval evaluations.
\end{itemize}

\smallbreak

Analogously we consider the eigenvalue problem \eqref{ev,nu} with $\nu\in (-\delta,0)$. The existence of non-trivial solution $\psi\in H^1(\R)$ is again  equivalent to the equation $0=g(k;\delta,\varepsilon)$ for $k\in (0,\sqrt{2\delta})$.
Hence there exists a unique zero $k\in (0,\sqrt{\delta})$.

To conclude, if $\varepsilon, \delta>0$ satisfy $\varepsilon+\delta<2$, then the operator $L$ has exactly two simple eigenvalues, located in $(-1,-1+\varepsilon]$ and $[1-\delta,1)$ respectively.

\end{ex}



\appendix

\setcounter{equation}{0}
\section{Characterization of eigenvalues when $u_+=1$ or $u_-=-1$ by a family of compact operators} 
In the appendix we consider the eigenvalue problem    
$$\mathcal{L}\Psi=\lambda\Psi,\quad \mathcal{L}=\left( \begin{matrix}
-u_- &  i\d_x \\  i\d_x & -u_+
 \end{matrix} \right): H^1(\R;\C^2)\to L^2(\R;\C^2),\quad \lambda\in (-1,1),$$ when the potentials $u_\pm$ satisfy
 \begin{equation}\label{CaseV+} \begin{split}
     &u_-=-1, \quad V_+=u_+-1\in L^\infty(\R;\R),\\
     &\hbox{such that }V_+(x)\rightarrow 0\hbox{ as }|x|\rightarrow\infty \hbox{ and }\|(V_+)^-\|_{L^\infty}<2,
 \end{split} 
 \end{equation}
 or 
 \begin{equation}\label{CaseV-} \begin{split}
     &u_+=1, \quad V_-=u_-+1\in L^\infty(\R;\R),\\
     &\hbox{such that }V_-(x)\rightarrow 0\hbox{ as }|x|\rightarrow\infty \hbox{ and }\|(V_-)^+\|_{L^\infty}<2.
 \end{split} 
 \end{equation} 
 If $u_-=-1$, then $\lambda+u_-\neq 0$ for all $\lambda\in (-1,1)$, and  by Lemma \ref{lem:cL} it is equivalent to consider the eigenvalue problem
 \eqref{lambda:phi,+} with $\lambda\in (-1,1)$:
 \begin{align}
          & \frac{1}{1-\lambda}(-\d_{xx}+1-\lambda^2)\phi + V_+ \phi=0. \label{ev,lambda,V+} 
    \end{align}
Similarly if $u_+=1$, then  $\lambda+u_+\neq 0$ for all $\lambda\in (-1,1)$, and it is equivalent to consider the eigenvalue problem
 \eqref{lambda:phi,-} with $\lambda\in (-1,1)$:
    \begin{align} \label{ev,lambda,V-} \frac{1}{1+\lambda}(-\d_{xx}+1-\lambda^2)\phi-V_-\phi=0.
    \end{align}

Theorem \ref{thm:cL} and Theorem \ref{thm:lambda} imply that in the case \eqref{CaseV+}, all the eigenvalues of \eqref{ev,lambda,V+} (and hence of $\cL$) inside $(-1,1)$ lie in $I_-=(-1, -1+\|(V_+)^-\|_{L^\infty}]$, and the number of  eigenvalues in $I_-$ is the same as the number of all  negative eigenvalues of \eqref{ev,mu,intro}:
\begin{align}
    &\label{ev,mu,V+}
       -\frac12\d_{xx}\psi + V_+\psi= \mu \psi, 
\end{align}
  counted by multiplicity,
  while for the case \eqref{CaseV-},  all  the eigenvalues of \eqref{ev,lambda,V-} (and hence of $\cL$) inside $(-1,1)$ lie in $I_+=[1-\|(V_-)^+\|_{L^\infty}, 1)$, and the number of  eigenvalues in $I_+$ is the same as the number of all  negative eigenvalues of \eqref{ev,nu,intro}:
    \begin{align}
     \label{ev,nu,V-}   
     -\frac12\d_{xx} \psi - V_-\psi= \nu \psi,
\end{align} 
counted by multiplicity.

In this appendix we are going to study the eigenvalue problems \eqref{ev,lambda,V+} and  \eqref{ev,lambda,V-}  for $\lambda\in (-1,1)$  directly, and we reformulate them as
\begin{align}\label{ev,compact,V+}
    (-\d_{xx}+1-\lambda^2)^{-1}V_+\phi=-\frac{1}{1-\lambda}\phi,
\end{align}
and
\begin{align}\label{ev,compact,V-}
    (-\d_{xx}+1-\lambda^2)^{-1}(-V_-)\phi=- \frac{1}{1+\lambda}\phi.
\end{align}
In Subsection \ref{subs:app1} we will introduce a symmetric and compact operator $K_\beta$, $\beta>0$ and study its negative eigenvalues.
We will then apply the obtained results to the eigenvalue problems \eqref{ev,compact,V+} and \eqref{ev,compact,V-}, to   characterize the eigenvalues $\lambda\in (-1,1)$ of $\cL$ when $u_-=-1$ or $u_+=1$ in Subsection \ref{subs:app2}.

\subsection{Analysis of a symmetric and compact operator $K_\beta$, $\beta>0$ on $H^1(\R)$}\label{subs:app1}
The above formulation \eqref{ev,compact,V+}-\eqref{ev,compact,V-} of the eigenvalue problems motivates us to define the operator 
 \begin{equation}\label{Kbeta}
 K_\beta=(-\d_{xx}+\beta)^{-1}V: H^1(\R)\to H^1(\R)
 \end{equation}
 where   $\beta>0$ is a fixed positive constant, and  $V\in L^\infty(\R)$ vanishes as infinity.

Then the operator $K_\beta$ is a symmetric and compact operator, when we endow $H^1(\R)$ with the inner product
 \begin{equation}\label{H1beta}
 \langle u, v\rangle_\beta:=\langle \d_x u, \d_x v\rangle_{L^2(\R)}+\beta\langle u, v\rangle_{L^2(\R)}.
 \end{equation}
It has an ONB of eigenfunctions and an associated eigenvalue sequence converging to $0$.
 We take all the negative eigenvalues
 \begin{equation}\label{mubeta}
 \{\gamma_j(\beta)\}_{j\in M}\subset (-\infty, 0),
 \end{equation}
 and order them non-decreasingly
 \begin{align*}
 |\gamma_1(\beta)|\geq  |\gamma_2(\beta)|\geq \cdots> 0.
 \end{align*}
 Here the set $M$ can be $\N$ or a finite set $\{1,\cdots,m\}$ or the empty set $\emptyset$.

 As before, we denote by $V^+$ resp. $V^-$ the positive resp. negative part of $V$,
 and we have the following properties of the set $M$ and the function $\gamma_j$, $j\in M$.
\begin{lem}[Properties of the set $M$ and the function  $\gamma_j$]\label{lem:N,mu}
The following holds true:
 \begin{enumerate}[(i)]
     \item The set $M=M^{(\beta)}$ is independent of $\beta\in \R^+$.
     \item For each $j\in M$, the function
     $$
     \gamma_j: (0,\infty)\to (-\infty, 0)
     $$
     is strictly increasing and   continuous, and for any $\beta\in (0,\infty)$,
     \begin{align}
       & |\gamma_j(\beta)|\leq \frac{\|V^-\|_{L^\infty}}{\beta}, \label{gamma,Linfty}
        \\
        &\limsup_{\tilde\beta\rightarrow\beta}\frac{|\gamma_j(\beta)-\gamma_j(\tilde\beta)|}{|\beta-\tilde\beta|}
         \leq \frac{\|V^-\|_{L^\infty}}{\beta^2}.\label{gamma,Lip}
     \end{align}
     Moreover, $\gamma_j$ is Lipschitz continuous on $[\delta,\infty)$, with Lipschitz constant $\delta^{-2}\|V^-\|_{L^\infty}$, for every $\delta>0$.
 \end{enumerate}
 \end{lem}
 \begin{proof} 
 This lemma follows by a similar argument as in the proof of Lemma \ref{lem:mu,nu}.
 For $\phi\in H^1(\R)\backslash\{0\}$ and $\beta>0$, we derive from \eqref{Kbeta} and \eqref{H1beta} that
 \begin{align*}
     \frac{\langle K_\beta \phi, \phi\rangle_\beta}{\langle \phi, \phi\rangle_\beta}
     =\frac{\langle V\phi, \phi\rangle_{L^2}}{\|\d_x\phi\|_{L^2}^2+\beta\|\phi\|_{L^2}^2}.
 \end{align*}
Poincar\'e's min-max principle (recalling \eqref{muj,lambda}-\eqref{nuj,lambda}) implies for any $j\in M^{(\beta)}$,
\begin{align}\label{minmax}
    \gamma_j(\beta)=\min_{U\subset H^1(\R)\hbox{ subspace},\, \dim(U)=j}\,\max_{\phi\in U\backslash\{0\}}
    \frac{\langle V\phi, \phi\rangle_{L^2}}{\|\d_x\phi\|_{L^2}^2+\beta\|\phi\|_{L^2}^2}.
\end{align}
The minimum is attained at
$$
U=U_j^{(\beta)}:=\hbox{\textrm{Span}}\{\phi_1(\beta), \cdots,\phi_j(\beta)\},
$$
where $\phi_l(\beta)$ denotes an eigenfunction of $K_\beta$ associated with $\gamma_l(\beta)$, and $\phi_1(\beta), \cdots, \phi_j(\beta)$ are chosen $\langle\cdot, \cdot\rangle_\beta$-orthonormal.

Let $\beta_1\in (0,\infty)$ and $j\in M^{(\beta_1)}$, so that $\gamma_j(\beta_1)<0$.
Then \eqref{minmax} implies
\begin{align*}\label{u<0}
    \frac{\langle V\phi, \phi\rangle_{L^2}}{\|\d_x\phi\|_{L^2}^2+\beta_1\|\phi\|_{L^2}^2}<0,
    \quad \forall \phi\in U_j^{(\beta_1)}\backslash\{0\},
\end{align*}
and thus $\langle V\phi, \phi\rangle_{L^2}<0$, and hence  
\begin{align} 
   \frac{ \langle V\phi, \phi\rangle_{L^2} }{\|\d_x\phi\|_{L^2}^2+\beta\|\phi\|_{L^2}^2}<0,
   \quad \forall \beta\in (0,\infty),
    \quad \forall \phi\in U_j^{(\beta_1)}\backslash\{0\}.
\end{align}
By the $L^2$-compactness of the unit sphere in the finite-dimensional space $U_j^{(\beta_1)}$, we conclude that
\begin{align*}
    \max_{\phi\in U_j^{(\beta_1)}\backslash\{0\}}
    \frac{\langle V\phi, \phi\rangle_{L^2}}{\|\d_x\phi\|_{L^2}^2+\beta\|\phi\|_{L^2}^2}
   = \max_{\phi\in U_j^{(\beta_1)}, \, \|\phi\|_{L^2}=1}
    \frac{\langle V\phi, \phi\rangle_{L^2}}{\|\d_x\phi\|_{L^2}^2+\beta }<0,
\end{align*}
and hence
\begin{align*}
     \min_{U\subset H^1(\R)\hbox{ subspace},\, \dim(U)=j}\,\max_{\phi\in U\backslash\{0\}}
    \frac{\langle V\phi, \phi\rangle_{L^2}}{\|\d_x\phi\|_{L^2}^2+\beta\|\phi\|_{L^2}^2}<0,\quad\forall \beta\in (0,\infty).
\end{align*}
This implies $j\in M^{(\beta)}$ for all $\beta\in (0,\infty)$, and we have proved $M^{(\beta_1)}\subset M^{(\beta)}$ for all $\beta_1, \beta\in (0,\infty)$.
The assertion (i) follows.

 Let $j\in M$ and $0<\beta<\beta_1<\infty$, so that 
 \begin{align*} 
    \frac{\langle V\phi, \phi\rangle_{L^2}}{\|\d_x\phi\|_{L^2}^2+\beta\|\phi\|_{L^2}^2}
   < \frac{\langle V\phi, \phi\rangle_{L^2}}{\|\d_x\phi\|_{L^2}^2+\beta_1\|\phi\|_{L^2}^2}<0,
   \quad \forall \phi\in U_j^{(\beta_1)}\backslash\{0\}.
 \end{align*}
 By the above compactness argument again we deduce
 \begin{align*} 
   \max_{\phi\in U_j^{(\beta_1)}\backslash\{0\}} \frac{\langle V\phi, \phi\rangle_{L^2}}{\|\d_x\phi\|_{L^2}^2+\beta\|\phi\|_{L^2}^2}
   < \max_{\phi\in U_j^{(\beta_1)}\backslash\{0\}} \frac{\langle V\phi, \phi\rangle_{L^2}}{\|\d_x\phi\|_{L^2}^2+\beta_1\|\phi\|_{L^2}^2} =\gamma_j(\beta_1).
 \end{align*}
 This implies $\gamma_j(\beta)<\gamma_j(\beta_1)$ by the min-max principle, and hence $\gamma_j:(0,\infty)\to (-\infty,0)$ is strictly increasing.
 
 Let $\beta\in (0,\infty)$ and we are going to show the  continuity of $\gamma_j$ at $\beta$.
 From the above argument we know that, for all $\phi\in U_j^{(\beta)}\backslash\{0\}$, $\langle V\phi, \phi\rangle_{L^2}<0$ and hence
 \begin{align*}
     \langle V^+\phi, \phi\rangle_{L^2}<\langle V^-\phi, \phi\rangle_{L^2},
 \end{align*}
 and thus
 \begin{align*}
     |\langle V\phi, \phi\rangle_{L^2}|=\langle V^-\phi, \phi\rangle_{L^2}-\langle V^+\phi, \phi\rangle_{L^2}\leq \|V^-\|_{L^\infty}\|\phi\|_{L^2}^2.
 \end{align*}
This implies 
\begin{align*}
     \frac{|\langle V\phi, \phi\rangle_{L^2}|}{\|\d_x\phi\|_{L^2}^2+\beta\|\phi\|_{L^2}^2}
     \leq \frac{\|V^-\|_{L^\infty}}{\beta},\quad\forall j\in M, \quad \forall \phi\in U_j^{(\beta)}\backslash\{0\},
\end{align*}
and hence the estimate  \eqref{gamma,Linfty} by the min-max principle \eqref{minmax}. 
Furthermore, we calculate for all $\tilde\beta\in (0,\infty)$ and all $\phi\in H^1(\R)\backslash\{0\}$ that
 \begin{align*}
     \Bigl| \frac{\langle V\phi, \phi\rangle_{L^2}}{\|\d_x\phi\|_{L^2}^2+\beta\|\phi\|_{L^2}^2}
       -\frac{\langle V\phi, \phi\rangle_{L^2}}{\|\d_x\phi\|_{L^2}^2+\tilde\beta\|\phi\|_{L^2}^2}\Bigr|
       &=\frac{|\langle V\phi, \phi\rangle_{L^2}| \cdot\|\phi\|_{L^2}^2 |\beta-\tilde\beta|}
       {(\|\d_x\phi\|_{L^2}^2+\beta\|\phi\|_{L^2}^2)\cdot(\|\d_x\phi\|_{L^2}^2+\tilde\beta\|\phi\|_{L^2}^2)}
       \\
       &\leq  \frac{|\langle V\phi, \phi\rangle_{L^2}| }{  \beta\|\phi\|_{L^2}^2  }
       \cdot \frac{ |\beta-\tilde\beta|}{ \tilde\beta }
     \leq   \frac{ \|V^-\|_{L^\infty} |\beta-\tilde\beta|}{\beta \tilde\beta}.
 \end{align*} 
 We immediately have  
 \begin{align*}
     \gamma_j(\beta)
     &= \max_{\phi\in U_j^{(\beta)}\backslash\{0\}}\frac{\langle V\phi, \phi\rangle_{L^2}}{\|\d_x\phi\|_{L^2}^2+\beta\|\phi\|_{L^2}^2}
     \\
     &\geq \max_{\phi\in U_j^{(\beta)}\backslash\{0\}}\frac{\langle V\phi, \phi\rangle_{L^2}}{\|\d_x\phi\|_{L^2}^2+\tilde \beta\|\phi\|_{L^2}^2}
     -\frac{\|V^-\|_{L^\infty}|\beta-\tilde\beta|}{\beta\tilde\beta}.
 \end{align*} 
 We apply the min-max principle in \eqref{minmax} to derive\begin{align*}
     \gamma_j(\beta)  
     \geq \gamma_j(\tilde\beta)-\frac{\|V^-\|_{L^\infty}|\beta-\tilde\beta|}{\beta\tilde\beta}.
 \end{align*} 
As $\beta, \tilde \beta\in (0,\infty)$ are chosen arbitrarily, we have similarly 
 \begin{align*}
     \gamma_j(\tilde\beta) 
     \geq \gamma_j(\beta)
     -\frac{\|V^-\|_{L^\infty}|\beta-\tilde\beta|}{\beta\tilde\beta}.
 \end{align*}  
Hence
 $$
 |\gamma_j(\beta)-\gamma_j(\tilde\beta)|\leq   \frac{\|V^-\|_{L^\infty}}{\beta\tilde\beta} |\beta-\tilde\beta|,
 $$
 and the estimate \eqref{gamma,Lip} follows.
 The Lipschitz continuity of $\gamma_j$ on $[\delta,\infty)$ with $\delta>0$ follows correspondingly.
 \end{proof}
 
 \subsection{Eigenvalues $\lambda\in (-1,1)$ of $\cL$ when $u_-=-1$ or $u_+=1$}\label{subs:app2}
 Assume \eqref{CaseV+}, and we now come back to the eigenvalue problem \eqref{ev,lambda,V+}, i.e. \eqref{ev,compact,V+} for $\lambda\in (-1,1)$.
For any $\beta>0$, let $\{\gamma_j^+(\beta)\}_{j\in M_+}$ denote all the negative eigenvalues of the symmetric and compact operator (recalling the definitions  \eqref{Kbeta}-\eqref{mubeta})
$$K_\beta^+=(-\d_{xx}+\beta)^{-1}V_+:H^1(\R)\to H^1(\R).$$ 
We define the even function
\begin{equation*} 
 \Gamma_j^+(\lambda):= \gamma_j^+(1-\lambda^2), \quad \lambda\in (-1,1),
\end{equation*} 
and we are indeed searching for $\lambda_j\in (-1,1)$ such that
\begin{equation}\label{gammaj,lambdaj}
 \Gamma_j^+(\lambda_j)=-\frac{1}{1-\lambda_j}.
\end{equation}

By Lemma \ref{lem:N,mu}, the   function  $\Gamma_j^+: (-1,1)\to (-\infty,0)$ 
 is  strictly increasing in $\lambda\in (-1,0)$ while strictly decreasing in $\lambda\in (0,1)$, and for any $\lambda\in (-1,1)$, using \eqref{gamma,Linfty} and \eqref{gamma,Lip},
\begin{align}
|\Gamma_j^+(\lambda)|\leq \frac{\|(V_+)^-\|_{L^\infty}}{1-\lambda^2},\label{Gamma,Linfty}
\\
    \limsup_{\tilde\lambda\rightarrow\lambda}
    \frac{|\Gamma_j^+(\lambda)-\Gamma_j^+(\tilde\lambda)|}{|\lambda-\tilde\lambda|}
   & =\limsup_{\tilde\lambda\rightarrow\lambda}
    \frac{|\gamma_j^+(1-\lambda^2)-\gamma_j^+(1-\tilde\lambda^2)|}{|\lambda^2-\tilde\lambda^2|} |\lambda+\tilde\lambda|\notag
    \\
&    \leq \frac{2|\lambda|\cdot \|(V_+)^-\|_{L^\infty}}{(1-\lambda^2)^2}.\label{Gamma,Lip}
\end{align}
If $\|(V_+)^-\|_{L^\infty}<2$, then for all $\lambda\in (0,1)$, we derive from \eqref{Gamma,Lip} that 
\begin{align*} 
  \limsup_{\tilde\lambda\rightarrow\lambda}
    \frac{|\Gamma_j^+(\lambda)-\Gamma_j^+(\tilde\lambda)|}{|\lambda-\tilde\lambda|} 
    <\frac{4 \lambda }{(1-\lambda^2)^2}<\frac{1}{(1-\lambda)^2},
\end{align*} 
that is,
\begin{align}\label{below}
  \limsup_{\tilde\lambda\rightarrow\lambda}
    \frac{|\Gamma_j^+(\lambda)-\Gamma_j^+(\tilde\lambda)|}{|\lambda-\tilde\lambda|} 
    < \Bigl| \frac{d}{d\lambda}\Bigl( \frac{1}{1-\lambda}\Bigr) \Bigr|,
    \quad \forall \lambda\in (0,1).
\end{align}  

Notice that the function $-\frac{1}{1-\lambda}: [-1,1)\to (-\infty,-\frac12]$ is a strictly decreasing function of $\lambda\in [-1,1)$, with its  maximum $-\frac12$ evaluated at $\lambda=-1$.
Let $M_+'$ denote the subset of $M_+$ defined by
\begin{equation}
    M_+'=\{j\in M_+\,|\, \gamma_j^+(0_+)<-\frac12 \}.
\end{equation}
We claim that for any $j\in M_+'$, there exists a unique solution $\lambda_j\in (-1,1)$  of \eqref{gammaj,lambdaj}.
Indeed, on one side, for any $j\in M_+'$, 
$$\Gamma_j^+((-1)_+)=\gamma_j^+(0_+)<-\frac12=\Bigl(-\frac{1}{1-\lambda}\Bigr)\Big|_{\lambda=-1}.$$
On the other side, there exists $\lambda_0\in [0,1)$ so that
$$
\Gamma_j^+(\lambda_0)\geq  -\frac{1}{1-\lambda_0},
$$
since, using \eqref{Gamma,Linfty}, if $\|(V_+)^-\|_{L^\infty}\leq 1$, we can take $\lambda_0=0$ so that  
     $$
     \Gamma_j^+(\lambda_0)=-|\Gamma_j^+(\lambda_0)|\geq -\|(V_+)^-\|_{L^\infty}\geq -1= -\frac{1}{1-\lambda_0},
     $$
     while if $\|(V_+)^-\|_{L^\infty}\in (1,2)$, we can take $\lambda_0=\|(V_+)^-\|_{L^\infty}-1\in (0,1)$ so that 
     $$
     \Gamma_j^+(\lambda_0)=-|\Gamma_j^+(\lambda_0)|\geq -\frac{1+\lambda_0}{1-(\lambda_0)^2}= -\frac{1}{1-\lambda_0}.
     $$
     By the continuity of the functions $\Gamma_j^+$ and $-\frac{1}{1-\lambda}$, there exists $\lambda_j\in (-1,\lambda_0]$ so that \eqref{gammaj,lambdaj} holds.
    The uniqueness of $\lambda_j\in (-1,1)$ follows from the strict monotonicity properties of the functions  $\Gamma_j^+|_{(-1,0)}$, $\Gamma_j^+|_{(0,1)}$ and $-\frac{1}{1-\lambda}|_{[-1,1)}$ and \eqref{below}.

This unique solution $\lambda_j\in (-1,1)$ of \eqref{gammaj,lambdaj} is an eigenvalue of \eqref{ev,lambda,V+}, i.e. \eqref{ev,compact,V+}, with the eigenspace the same as the one associated with the negative eigenvalue $\gamma_j^+(1-\lambda_j^2)$ of the operator $K_{1-\lambda_j^2}^+$.
 
 Conversely, let $(\lambda, \phi)\in (-1,1)\times (H^1(\R)\backslash\{0\})$ denote any eigenpair of the eigenvalue problem   \eqref{ev,lambda,V+} and hence of \eqref{ev,compact,V+}, then $\phi$ is an eigenfunction associated with the (negative) eigenvalue $-\frac{1}{1-\lambda}\in (-\infty, -\frac12)$   of the operator $K_{1-\lambda^2}^+$.
 Thus there exists $j\in M_+$ so that \eqref{gammaj,lambdaj} holds:
 $$
 \gamma_j^+(1-\lambda^2)=-\frac{1}{1-\lambda}.
 $$
 Since $\gamma_j^+: (0,\infty)\to (-\infty,0)$ is strictly increasing, we have indeed $j\in M_+'$:
 \begin{align*}
     \gamma_j^+(0_+)<\gamma_j^+(1-\lambda^2)=-\frac{1}{1-\lambda}<-\frac12.
 \end{align*}
 
The case \eqref{CaseV-} can be considered analogously: We replace $V_+$  by $-V_-$, and for any $\beta>0$ we denote by $\{\gamma_j^-(\beta)\}_{j\in M_-}$   all the negative eigenvalues of the operator $K_\beta^-=(-\d_{xx}+\beta)^{-1}(-V_-):H^1(\R)\to H^1(\R).$   
 Let $\Gamma_j^-(\lambda)=\gamma_j^-(1-\lambda^2)$ for $\lambda\in (-1,1)$, and we search for $\lambda_j\in (-1,1)$ such that
 \begin{align}\label{gammaj,lambdaj,-}
     \Gamma_j^{-}(\lambda_j)=-\frac{1}{1+\lambda_j}.
 \end{align} 
For any $j\in M_-'=\{j\in M_-\,|\, \gamma_j^-(0_+)<-\frac12\}$, the existence of a unique solution $\lambda_j\in (-1,1)$ follows similarly as for the case \eqref{CaseV+}: The function $\Gamma_j^{-}: (-1,1)\mapsto (-\infty,0)$ is continuous, strictly increasing in $\lambda\in (-1,0)$, strictly decreasing in $\lambda\in (0,1)$ and bounded as $|\Gamma_j^-(\lambda)|\leq \frac{\|(V_-)^+\|_{L^\infty}}{1-\lambda^2}$, while  the function $-\frac{1}{1+\lambda}: (-1,1]\to (-\infty, -\frac12]$ is continuous and strictly increasing, with its maximum $-\frac12$ attained at $\lambda=1$, and they satisfy 
   \begin{align*} 
   & \Gamma_j^{-}(1_-)<-\frac12,\quad \Gamma_j^-(\lambda_0)\geq -\frac{1}{1+\lambda_0}\hbox{ for some }\lambda_0\in (-1,0],
   \\
&  \limsup_{\tilde\lambda\rightarrow\lambda}
    \frac{|\Gamma_j^-(\lambda)-\Gamma_j^-(\tilde\lambda)|}{|\lambda-\tilde\lambda|} 
    < \Bigl| \frac{d}{d\lambda}\Bigl( \frac{1}{1+\lambda}\Bigr) \Bigr|,
    \quad \forall \lambda\in (-1,0).
\end{align*}  
This unique solution $\lambda_j\in (-1,1)$ of \eqref{gammaj,lambdaj,-} is an eigenvalue of \eqref{ev,lambda,V-}, i.e. \eqref{ev,compact,V-}, with the eigenspace the same as the one associated with the negative eigenvalue $\gamma_j^-(1-\lambda_j^2)$ of the operator $K_{1-\lambda_j^2}^-$.
Conversely, for any eigenpair $(\lambda,\phi)\in (-1,1)\times H^1\backslash\{0\}$ of \eqref{ev,lambda,V-}, i.e. \eqref{ev,compact,V-}, $\phi$ is an eigenfunction associated with the eigenvalue $-\frac{1}{1+\lambda}\in (-\infty, -\frac12)$ of the operator $K_{1-\lambda^2}^-$, and hence there exists $j\in M_-'$ so that $\gamma_j^-(1-\lambda^2)=-\frac1{1+\lambda}$.

 To conclude, we have proved
 \begin{thm}[Eigenvalues inside $(-1,1)$ of $L,\cL$  when $u_-=-1$ or $u_+=1$]\label{thm:V-}
 Let $q(x)$ satisfies \eqref{Assumption}-\eqref{BC:q,varraphi}-\eqref{smallness:upm}.
 The following holds true:
 \begin{itemize}
     \item  If $u_-=-1$ and   $\{\gamma_j^+(\beta)\}_{j\in M_+'}$, $\beta>0$ denote all the  negative eigenvalues of the operator $K_{\beta}^+=(-\d_{xx}+\beta)^{-1}(u_+-1)$ such that $\gamma_j^+(0_+)<-\frac12$,
     then for all $j\in M_+'$, there exists a unique fixed point $\lambda_j\in (-1,1)$ of the mapping 
     $1+\frac1{\gamma_j^+(1-\lambda^2)}$, and $\{\lambda_j\}_{j\in M_+'}$ is the    set of eigenvalues of $L, \cL$ in  $(-1,1)$, with the eigenspace of $\lambda_j$ coinciding with the eigenspace associated with the negative eigenvalue $\gamma_j^+(1-\lambda_j^2)$ of the operator $K_{1-\lambda_j^2}^+$.
  \item If $u_+=1$ and   $\{\gamma_j^-(\beta)\}_{j\in M_-'}$, $\beta>0$ denote all the  negative eigenvalues of the operator $K_{\beta}^-=(-\d_{xx}+\beta)^{-1}(-u_--1)$ such that $\gamma_j^-(0_+)<-\frac12$,
     then for all $j\in M_-'$, there exists a unique fixed point $\lambda_j\in (-1,1)$ of the mapping 
     $-1-\frac1{\gamma_j^-(1-\lambda^2)}$, and $\{\lambda_j\}_{j\in M_-'}$ is the    set of eigenvalues of $L, \cL$ in  $(-1,1)$, with the eigenspace of $\lambda_j$ coinciding with the eigenspace associated with the negative eigenvalue $\gamma_j^-(1-\lambda_j^2)$ of the operator $K_{1-\lambda_j^2}^-$.  
 \end{itemize} 
 \end{thm} 

 \section*{Acknowledgments}
The authors are funded by the Deutsche Forschungsgemeinschaft (DFG, German Research 
Foundation) – Project-ID 258734477 – SFB 1173.
We thank Kaori Nagatou for her support on the computer-assisted proof in Example 3.4.

 \end{document}